\documentclass[10pt]{article}
\usepackage{graphicx}
\usepackage{amsmath,amssymb,amsthm,amsfonts}
\usepackage{mathrsfs}

\usepackage[toc,page]{appendix}

\usepackage{amssymb}
\usepackage{varioref} 

\newtheorem{thm}{Theorem}[section]

\newtheorem{lem}[thm]{Lemma}

\theoremstyle{definition}
\newtheorem{defn}{Definition}[section]

\numberwithin{equation}{section}

\DeclareMathSymbol{\C}{\mathalpha}{AMSb}{"43}

\newcommand{\lam}{\lambda}

\newcommand{\bsub}{\begin{subequations}}
\newcommand{\esub}{\end{subequations}$\!$}

\begin{document}

\title{Constraint Minimization Problem of the Nonlinear Schr\"{o}dinger Equation with the Anderson Hamiltonian}

\author{ Qi Zhang\thanks{Yanqi Lake Beijing Institute of Mathematical Sciences and Applications, Beijing 101408, China and Yau Mathematical Sciences Center, Tsinghua University, Beijing 100084, China, Email:qzhang@bimsa.cn}, Jinqiao Duan\thanks{Department of Mathematics and Department of Physics, Great Bay University, Dongguan, Guangdong 523000, China, Email:duan@gbu.edu.cn}}


\smallbreak \maketitle

\begin{abstract}

We consider the two-dimensional nonlinear Schr\"{o}dinger equation with a Gaussian white noise potential, described by the Anderson hamiltonian. After define the corresponding energy space via the paracontrolled distribution framework from singular stochastic partial differential equations, we prove the existence of the minimizer as the least energy solution by studying a minimization problem of the corresponding energy functional subject to $L^2$ constraints. Subsequently, we establish $L^2$ and Schauder estimates for the minimizer, which is a weak solution of the stationary nonlinear Schr\"{o}dinger equation. Finally, we derive a tail estimate for the distribution of the principal eigenvalue corresponding to the least energy solution by energy estimates.

\end{abstract}

\vskip 0.2truein


\vskip 0.2truein

\section{Introduction}

We consider the stationary nonlinear Schr\"{o}dinger equation with a singular Gaussian white noise potential on the $2$-dimensional torus $\mathbb{T}^2$, given by
\begin{equation}\label{NAME}
    - \Delta u(x) + u(x) - \xi\diamond u(x) + a|u(x)|^2 u(x) = \lambda u(x), \quad x\in \mathbb{T}^2,
\end{equation}
where $a$ is a positive parameter,  $\xi$ is a Gaussian white noise on $\mathbb{T}^2$, the Wick product is denoted by $\diamond$, and $\lambda$ is the principal eigenvalue.

The corresponding random Schr\"{o}dinger operator $\mathscr{H}u:= -\Delta u + u - \xi\diamond u$, is also known as the Anderson hamiltonian, which was originally introduced by Anderson \cite{A1958}. In Anderson's seminal work, the Schr\"{o}dinger equation with the discrete Anderson hamiltonian $\mathscr{H}$ on the lattice $\mathbb{Z}^d$ was used to describe the wave propagation of quantum particles through random disordered media, and is related to Anderson localization in condensed matter physics. Anderson localization is characterized by the existence of localised eigenfunctions, with the principal eigenvalue of the Anderson hamiltonian $\mathscr{H}$ playing a key role. The discrete Anderson hamiltonian and associated Anderson model have been extensively studied in the past few decades, as seen in surveys such as \cite{CM1994, K2015} and references therein.

The continuous Anderson hamiltonian that we consider here arises as the continuum limit of discrete Anderson hamiltonian. Since the $d$-dimensional Gaussian white noise $\xi$ is a random Schwartz distribution with H\"{o}lder regularity index slightly below $-d/2$, the random potential is singular and the construction of the Anderson hamiltonian is not simple. In dimension 1, Fukushima and Nakao \cite{FN1977} constructed the Anderson hamiltonian via Dirichlect from. The asymptotic behaviour of its eigenvalues and eigenfunctions was studied by McKean \cite{M1994}.

When $d\geq 2$, the Gaussian white noise potential in Anderson hamiltonian is more singular. Thus the classical method for the construction of the Anderson hamiltonian is not working, and the renormalization argument is needed. In recent years, some new mathematical theories from the field of singular stochastic partial differential equations, such as regularity structures by Harier \cite{H2014} or paracontrolled distributions by Gubinelli, Imkeller and Perkowski \cite{GIP2015}, have been developed to carry out such a construction. Both regularity structures and paracontrolled distribution frameworks are developed from the theory of controlled rough paths, and allow a pathwise description of the singular stochastic partial differential equations. In particular, the paracontrolled distribution relies on harmonic analysis tools, including Littlewood-Paley decomposition, Besov space, and Bony's paraproducts. It is natural to generalize some classical partial differential equation methods to study singular stochastic partial differential equations in the paracontrolled distribution framework.

The construction of the Anderson Hamiltonian on $\mathbb{T}^2$ using paracontrolled distributions was studied by Allez and Chouk \cite{AC2015}. Labb\'{e} \cite{L2019} utilized the tools of regularity structures to construct the Anderson Hamiltonian with both periodic and Dirichlet boundary conditions for $d\leq 3$. Chouk and van Zuijlen \cite{CV2021} considered the asymptotics of the eigenvalues of the Anderson Hamiltonian. Gubinelli, Ugurcan and Zachhuber \cite{GUZ2020} extended the construction of the Anderson Hamiltonian to $\mathbb{T}^3$ using paracontrolled distributions. The well-posedness of the continuous parabolic Anderson model equation has also been studied in \cite{GIP2015, H2014, HL2015} using different approaches, including regularity structures, paracontrolled distributions, and the transformation method with an elaborate renormalization procedure.

In physics, the nonlinear Schr\"{o}dinger equation is a crucial model that appears in various branches of physics, particularly in Bose-Einstein condensation \cite{DGPS1999}. The nonlinear term in this equation approximates the mean field interaction between atoms in many-body particle systems. Recently, the theory of Anderson localization has been extended to many-body particle systems, and this phenomenon can be modeled by the nonlinear Schr"{o}dinger equation with the Anderson Hamiltonian \cite{AW2009, BAA2006}. While the Anderson localization of the discrete nonlinear Schr\"{o}dinger equation with Anderson Hamiltonian has been considered by \cite{CSZ2021}, and the well-posedness of the continuous nonlinear Schr\"{o}dinger equation with Anderson Hamiltonian has been studied in \cite{GUZ2020}, the study of Anderson localization for the continuous nonlinear Schr\"{o}dinger equation remains largely unexplored.

In this paper, we focus on the minimization problem with $L^2$ constraints of the energy functional associated with the nonlinear Schr\"{o}dinger equation (\ref{NAME}): 
\begin{equation}\label{EAH}
    E(w):=  \frac{1}{2}B_{\mathscr{H}}(w,w) + \frac{a}{4}\int_{\mathbb{T}^2}|w(x)|^4dx, \quad w \in \mathscr{D}^{\alpha,1}_{\vartheta},
\end{equation}
where $B_{\mathscr{H}}(w,w)$ is quadratic form associated with the Anderson hamiltonian $\mathscr{H}$, and the domain  $\mathscr{D}^{\alpha,1}_{\vartheta}$ is the energy space which will be defined in Definition \ref{Domina12}. For the explicit definition of the energy functional $E(u)$, we refer to Definition \ref{thmBh}. We will show that the minimization problem (\ref{EAH}) has minimizer $u$, which is a least energy solution of the stationary stochastic nonlinear Schr\"{o}dinger equation (\ref{NAME}). Moreover, the principle eigenvalue $\lambda$ of (\ref{NAME}) has the following variational representation:
\begin{equation}\label{GS}
    \lam = \inf_{w \in \mathscr{D}^{\alpha,1}_{\vartheta}, \|w\|_{L^2}=1} E(w) = E(u).
\end{equation}

Even though the Anderson hamiltonian $\mathscr{H}$ in high dimensional is well-defined, the energy variation problems associate with the Anderson hamiltonian are not well studied. As we known, this is the first attempt to consider the minimization problem associated with singular SPDEs under the paracontrolled distribution framework. We also remark that in \cite{GUZ2020}, the energy space and energy functional associated with the Anderson hamiltonian $\mathscr{H}$ are defined by another way. 

By the direct method of calculus of variation, we prove the existence of the minimizer as the first main result in this paper.

\begin{thm}\label{Ex}
(Existence of the minimizer) There exists at least one minimizer $u\in \mathscr{D}_{\vartheta}^{\alpha,1}$ for the energy functional $E(u)$. Moreover, the minimizer $u$ is a weak solution of the elliptic singular stochastic partial differential equation (\ref{NAME}).
\end{thm}

We refer subsection \ref{PfEx} for the proof of Theorem \ref{Ex}. In this proof, we employ the high-low frequency decomposition, and split the singular terms $\xi\diamond u = \Phi(u) + \Psi(u)$, where $\Phi(u)$ contains all of irregular but linear terms, and $\Psi(u)$ contains all the regular terms. Then the weak solution $u$ are paracontrolled by $\vartheta$ with reminders $R(u)$ and $u^{\sharp}$, i.e. $u = u\prec\vartheta + R(u) + u^{\sharp}$ (see Definition \ref{defParaC} and \ref{thmBh}).

Moreover, the original singular stochastic partial differential equation in the following elliptic system:
\begin{equation}\label{decGEAM1}
  \left\{
   \begin{aligned}
   &  -\Delta (u\prec \vartheta + R(u)) + (u\prec \vartheta + R(u)) = \Phi(u) + \lambda (u\prec \vartheta + R(u)), \\
   &  -\Delta u^{\sharp} + u^{\sharp} = a|u|^2 u+ \Psi(u) + \lambda u^{\sharp}.
   \end{aligned}
   \right.
\end{equation}
The second main result in this paper is an H\"{o}lder regularity via $L^2$ estimates and Schauder estimates for the above elliptic system, see Theorem \ref{L2E} and \ref{SEu}.

\begin{thm}\label{Re}
(Regularity of the minimizer) The minimizer $u \in \mathscr{C}^{\alpha}$, with remainder terms $R(u)\in \mathscr{C}^{2\alpha}$ and $u^{\sharp}\in \mathscr{C}^{3\alpha}$,  where 
 $\mathscr{C}^{\alpha}:=B^{\alpha}_{\infty,\infty}$ denotes the H\"{o}lder-Besov space.
\end{thm}

Our third main result is an estimate on the left tail of the distributions of the principle eigenvalue. This result is coming from the variational representation of the principle eigenvalue (\ref{GS}) and some energy estimates.

\begin{thm}\label{TEGS}
(Tail estimate of the principle eigenvalue) Let $\lambda$ the principle eigenvalue of the stationary singular stochastic nonlinear Schr\"{o}dinger equation (\ref{NAME}). Then there exists two constants $C_{2}>C_{1}>0$ so that for all $x>0$ large enough we have
\begin{equation}\label{TE}
    e^{C_{2} -x} \leq \mathbb{P}( \lambda \leq -x) \leq e^{C_{1} -x}.
\end{equation}
\end{thm}

Throughout the paper, we use the notation $a\lesssim b$ if there exists a constant $C>0 $, independent of the variables under consideration, so that $a \leq C\cdot b$, and we denote $a\simeq b$ if $a\lesssim b$ and $b \lesssim a$. The Fourier transform on the torus $\mathbb{T}^d$ is defined with $\hat{f}(k):=(\mathscr{F}_{\mathbb{T}^d}f)(k)= \int_{\mathbb{T}^d}e^{-2\pi k\cdot x}f(x)dx$, and the inverse Fourier transform on the torus $\mathbb{T}^d$ is given by $(\mathscr{F}^{-1}_{\mathbb{T}^d}\hat{f})(x) = \sum_{k \in \mathbb{Z}^d}e^{2\pi i k \cdot x}\hat{f}(k)$. We denote the space of Schwartz functions on $\mathbb{T}^d$ by $\mathcal{S}(\mathbb{T}^d)$. The space of tempered distributions on $\mathbb{T}^d$ is denoted by $\mathcal{S}'(\mathbb{T}^d)$. We also denote $\mathscr{L}:=  -\Delta +1$. For Besov spaces and Sobolev space, we write $\mathscr{C}^{\alpha}:=B^{\alpha}_{\infty,\infty}$ as the H\"{o}lder-Besov space, and $H^{\alpha} := B^{\alpha}_{2,2}$ as the Sobolev space. For Bony's paraproducts, we use the notations $\prec$ and $\succ$ and for the resonant product we use $\circ$. See the Appendix A for precise definitions of Besov spaces and Bony's paraproducts.

The rest of this paper is organized as follows: In Section \ref{S2}, we revisit some basic notations and results of the Anderson hamiltonian $\mathscr{H}$, and introduce the high-low frequency decomposition. In Section \ref{S3}, we define the energy functional $E(u)$ by the quadratic form of the Anderson hamiltonian on the modified paracontrolled space $\mathscr{D}^{\alpha,1}_{\vartheta}$, and show that the energy functional is a $C^1$ map from $\mathscr{D}^{\alpha,1}_{\vartheta}$ to $\mathbb{R}$, and the singular stochastic partial differential equation (\ref{NAME}) is the corresponding Euler-Lagrange equation of energy functional $E(u)$. We also obtain the existence of minimizer as a weak solution of the singular stochastic partial differential equation (\ref{NAME}) by direct method in the calculus of variations. In Section \ref{S4}, we decomposed the singular stochastic partial differential equation (\ref{NAME}) into a simpler elliptic system, and establish the $L^2$ estimates and Schauder estimates for the minimizer $u$.  This paper ends with some summary and discussion in Section \ref{DC}.

\section{Preliminaries}\label{S2}

\subsection{Renormalization and paracontrolled distributions}

Here, we introduce the renormalization argument associated with the spatial white noise $\xi$ on $\mathbb{T}^2$. The spatial white noise $\xi$ is given as follows.
\begin{defn}
Let $(\hat{\xi}(k))_{k\in \mathbb{Z}^2}$ be a sequence of i.i.d. centered complex Gaussian random variables defined on a complete probability space $(\Omega, \mathscr{F}, \mathbb{P})$ with covariance
\begin{equation*}
    \mathbb{E}(\hat{\xi}(k)\Bar{\hat{\xi}}(l)) = \delta(k-l),
\end{equation*}
and $\hat{\xi}(k) = \Bar{\hat{\xi}}(-k)$. Then the spatial white noise $\xi$ on $\mathbb{T}^2$ can be defined as the sum of random series
\begin{equation*}
    \xi(x) =\sum_{k \in \mathbb{Z}^2} \hat{\xi}(k)e^{2\pi ik \cdot x}.
\end{equation*}    
\end{defn}
Then the spatial white noise $\xi$ on $\mathbb{T}^2$ is a centered Gaussian process with value in $\mathcal{S}'(\mathbb{T}^2)$ so that for all $f,g\in \mathcal{S}(\mathbb{T}^2)$, we have $\mathbb{E}[\xi(f)\xi(g)]=\langle f,g\rangle_{L^2 (\mathbb{T}^2)}$.
By Kolmogorov's continuity criterion, the spatial white noise $\xi$ take value in $\mathscr{C}^{-1-\kappa}$ for an arbitrary small $\kappa>0$. Since $\xi$ is only a distribution, the product $u\xi$ is ill-defined in classic sense. How to let singular term $u\xi$ make sense is a main challenge in studying the Anderson hamiltonian. We set a smooth approximation $\xi_{\epsilon}$ of $\xi$. More precisely, we set $\varphi: \mathbb{T}^2 \rightarrow \mathbb{R}^{+}$ be a smooth function with $\int_{\mathbb{T}^2}\varphi dt =1$, and define $\xi^{\epsilon} = \epsilon^{-2}\varphi(\epsilon\cdot)\ast\xi$ for $\epsilon >0$ as the mollification of $\xi$. Now we take 
\begin{equation*}
    \vartheta = (-\Delta + 1)^{-1}\xi = \mathscr{L}^{-1}\xi,
\end{equation*}
Then $\|\vartheta\|_{1-\kappa} \lesssim \|\xi\|_{-1-\kappa}$.
In order to obtain a well-defined area $\vartheta \circ \xi$, we have to renormalize the product by “subtracting an infinite constant” as following arguments (See \cite[Lemma 5.8]{GIP2015}).

\begin{lem}\label{renormalize}
If $\vartheta_{\epsilon} = (-\Delta + 1)^{-1}\xi_{\epsilon}$, then
\begin{equation*}
    \lim_{\epsilon\rightarrow 0}\mathbb{E}[\|\vartheta\diamond\xi-(\vartheta_{\epsilon}\circ\xi_{\epsilon}-C_{\epsilon})\|^p_{\mathscr{C}^{-2\kappa}}]=0
\end{equation*}
for all $p\geq 1$ and $\kappa >0$ with the renormalization constant
\begin{equation*}
    C_{\epsilon}= \mathbb{E}(\vartheta_{\epsilon}\circ\xi_{\epsilon})  = \sum_{k\in \mathbb{Z}^2} \frac{|\mathscr{F}_{\mathbb{T}^2}\varphi(\epsilon k)|^2}{|k|^2+ 1} .
\end{equation*}
\end{lem}

By Bony's paraproducts, we define the following paracontrolled ansatz:
\begin{defn}\label{DefD}
Let $\alpha\in (2/3,1)$ and $\beta \in (0,\alpha]$ be such that $2\alpha+\beta >2$. We say a pair $(u, u')\in H^{\alpha}\times H^{\beta}$ is called paracontrolled by $\vartheta$ if $u^{\vartheta}:= u-u'\prec \vartheta \in H^{2\alpha}$.
\end{defn}

\subsection{The Anderson hamiltonian}

We recall some basic notations and results of the Anderson hamiltonian on $\mathbb{T}^2$ under the paracontrolled distribution framework from \cite{AC2015, GUZ2020}. The domain of the Anderson hamiltonian is given by the paracontrolled space $\mathscr{D}^{\alpha}_{\vartheta}$.

\begin{defn}\label{defParaC}
Let $\alpha \in (2/3,1)$. We define the space of distributions which are paracontrolled by $\vartheta$,
 \begin{equation}
    \mathscr{D}^{\alpha}_{\vartheta}:= \left\{ u\in H^{\alpha},  u^{\vartheta}:= u- u\prec\vartheta \in H^{2\alpha} \right\}.
\end{equation}
The space $\mathscr{D}_{\vartheta}^{\alpha}$ equipped with the inner product
\begin{equation}
    \langle u, v \rangle_{\mathscr{D}_{\vartheta}^{\alpha}}:= \langle u, v \rangle_{H^{\alpha}} + \langle u^{\vartheta}, v^{\vartheta} \rangle_{H^{2\alpha}}, \quad  u,v\in \mathscr{D}_{\vartheta}^{\alpha}.
\end{equation}   
is a Hilbert space.
\end{defn}

Then for every $u\in \mathscr{D}^{\alpha}_{\vartheta}$, by Bony's paraproduct estimates and the commutator estimate, the singular term  $u\diamond\xi$ is defined as
\begin{align}\label{Wickp1}
   u\diamond\xi = &  u\prec \xi + u\succ \xi+ u\circ\xi \nonumber \\
    = & u\prec \xi + u\succ \xi +(u\prec\vartheta)\circ\xi + R(u)\circ\xi+ u^{\vartheta}\circ\xi \nonumber \\
              = &  u\prec \xi + u\succ \xi  + C(u,\vartheta,\xi) +u (\vartheta \diamond \xi) + R(u)\circ\xi +u^{\vartheta}\circ\xi.
\end{align}
Moreover, the the Anderson hamiltonian $\mathscr{H}$ can be defined as follows 
\begin{equation}
    \mathscr{H} u:= -\Delta u +u -u\diamond \xi =  -\Delta u +u - u\prec\xi   - u\succ\xi - C(u,\vartheta,\xi) - u(\vartheta\diamond\xi) - u^{\vartheta}\circ\xi
\end{equation}

With this definition of the Anderson hamiltonian $\mathscr{H}$, the authors of \cite{AC2015} prove the following Theorem.

\begin{thm}\label{SAOah}
The Anderson hamiltonian $\mathscr{H}$ is a self-adjoint operator from $\mathscr{D}_{\vartheta}^{\alpha}$ to $L^2$ with pure point spectrum $(\Lambda_{k})_{k\geq 1}$ with the corresponding eigenfunctions $(e_{k})_{k\geq 1}$  such that
\begin{equation*}
    \Lambda_{1} < \Lambda_{2} < \Lambda_{3} < \dots \quad \text{almost surely.}
\end{equation*}
\end{thm}
The spectrum of the Anderson hamiltonian $\mathscr{H}$ has the following tail estimates, see e.g. \cite[Proposition 5.4]{AC2015}.
\begin{thm}
For all $k \in\mathbb{N}$, there exists two constants $C_{2} \geq C_{1}>0$ so that for all $x>0$ large enough,
\begin{equation}\label{TElam}
    e^{C_{2} -x} \leq \mathbb{P}( \Lambda_{k} \leq -x) \leq e^{C_{1} -x}.
\end{equation}
\end{thm}

\subsection{Localization  operators and high-low frequency decomposition}

Now  we  define  localization  operators $\mathscr{U}^{N,\gamma}_{\leq} $, $\mathscr{U}^{N,\gamma}_{>}$ for the high-low frequency decomposition. For every $f \in \mathcal{S}^{\prime}(\mathbb{T}^d)$, we define the following localization operators
\begin{equation}
    \mathscr{U}^{N,\gamma}_{\leq} f= \sum_{-1 \leq j \leq N}\Delta_{j} f + \sum_{j > N} 2^{- j \gamma} \Delta_{j} f, \quad  \mathscr{U}^{N,\gamma}_{>} f=  \sum_{j > N}(1-2^{- j \gamma})\Delta_{j} f.
\end{equation}

\begin{lem}\label{Localization}
Let $N, \gamma>0$ and $f \in \mathcal{S}^{\prime}(\mathbb{T}^d)$. Then for every $\alpha, \delta>0$ and $\beta \in [0,\gamma],$ we have
\begin{equation*}
   \left\|\mathscr{U}^{N,\gamma}_{>} f\right\|_{\mathscr{C}^{-\alpha-\delta}} \lesssim 2^{-\delta N}\|f\|_{\mathscr{C}^{-\alpha} }, \quad \|\mathscr{U}^{N,\gamma}_{\leq} f\|_{\mathscr{C}^{-\alpha+\beta}} \lesssim 2^{\beta N}\|f\|_{\mathscr{C}^{-\alpha} }.
\end{equation*}
\end{lem}
\begin{proof}
We estimate
\begin{align}
    \|\mathscr{U}^{N,\gamma}_{>} f \|_{\mathscr{C}^{-\alpha-\delta}} = & \sup_{l \geq -1} \left[ 2^{l(-\alpha - \delta)} \|\Delta_{l} (\sum_{j > N}(1-2^{- j \gamma})\Delta_{j} f) \|_{L^{\infty}}\right] \nonumber\\
    \leq & 2^{-\delta N}\sup_{l \geq -1} \left[ 2^{-\alpha l} \|\Delta_{l} f \|_{L^{\infty}}\right] \nonumber\\
    \leq & 2^{-\delta N} \|f\|_{\mathscr{C}^{-\alpha} }
\end{align}
Using same argument, we also have
\begin{align}
    \|\mathscr{U}^{N,\gamma}_{\leq} f\|_{\mathscr{C}^{-\alpha+\beta}} = & \sup_{l \geq -1} \left[ 2^{l(-\alpha + \beta)} \|\Delta_{l}(\sum_{-1 \leq j \leq N}\Delta_{j} f + \sum_{j \geq N} 2^{- j \gamma} \Delta_{j} f ) \|_{L^{\infty}}\right] \nonumber\\
     \leq & 2^{\beta N} \|f\|_{\mathscr{C}^{-\alpha} }.
\end{align}

\end{proof}

\section{Energy functional with the Anderson hamiltonian}\label{S3}

\subsection{The energy space of the Anderson hamiltonian}

In order to investigate the minimization problem of the energy functional, define a suitable energy space as the domain of energy functional is essential. If the potential in the Schr\"{o}dinger operator is regular, the classical $H^{1}$ space is an appropriate space associate with the Laplace operator. From this point of view, the original paracontrolled space $\mathscr{D}^{\alpha}_{\vartheta}$ is too regular, and inappropriate to be the energy space.

In order to find the suitable energy space of the energy functional $E(u)$ associated with the Anderson hamiltonian $\mathscr{H}$, we split the reminder $u^{\vartheta} = R(u) + u^{\sharp}$, and use localization operators $\mathscr{U}^{N,\gamma}_{\leq} $, $\mathscr{U}^{N,\gamma}_{>}$ with different parameters to split singular terms $\xi\diamond u = \Phi(u) + \Psi(u)$, where \begin{align*}
    \Phi(u) := &  u \prec \mathscr{U}^{L,\gamma_{1}}_{>} \xi + u \succ \mathscr{U}^{L,\gamma_{1}}_{>} \xi + u \succ \mathscr{U}^{K,\gamma_{2}}_{>}(\vartheta\diamond\xi) + u \prec\mathscr{U}^{K,\gamma_{2}}_{>}(\vartheta\diamond\xi), \\
    \Psi(u):= &  u^{\sharp}\circ\xi + R(u)\circ\xi + C(u,\vartheta,\xi)+ u\prec\mathscr{U}^{K,\gamma_{2}}_{\leq}(\vartheta\diamond\xi) + u\succ\mathscr{U}^{K,\gamma_{2}}_{\leq}(\vartheta\diamond\xi) + u \circ(\vartheta\diamond\xi)\\
     & + u\prec \mathscr{U}^{L,\gamma_{1}}_{\leq} \xi + u\succ \mathscr{U}^{L,\gamma_{1}}_{\leq}\xi .
\end{align*}
Here, we choose different parameters in $\mathscr{U}^{N,\gamma}_{\leq} $, $\mathscr{U}^{N,\gamma}_{>}$, so that $\Phi$ is the collection of all terms of negative regularity, and $\psi$ the collection of all other regular terms (belonging to $L^{2}$). Now we define the following modified paracontrolled space as the energy space of the Anderson hamiltonian $\mathscr{H}$. 

\begin{defn}\label{Domina12}
Let $\alpha \in (2/3,1)$, $\beta \in \{1,2\}$. We define
\begin{equation}
    \mathscr{D}_{\vartheta}^{\alpha,\beta}:= \left\{ u\in H^{\alpha}: \text{there exists }  u^{\sharp}:= u- u\prec\vartheta - R(u) \in H^{\beta} \right\},
\end{equation}
where $R(u)$ is a linear operator which given by $R(u)= \mathscr{L}^{-1}( \Phi(u) - \mathscr{L}(u \prec \vartheta))$.
The space $\mathscr{D}_{\vartheta}^{\alpha,\beta}$ is equipped with the inner product
\begin{equation}
    \langle u, v \rangle_{\mathscr{D}_{\vartheta}^{\alpha,\beta}}:= \langle u, v \rangle_{H^{\alpha}} + \langle u^{\sharp}, v^{\sharp} \rangle_{H^{\beta}}, \quad  u,v\in \mathscr{D}_{\vartheta}^{\alpha,\beta}.
\end{equation}
\end{defn}
Then for every $u\in\mathscr{D}_{\vartheta}^{\alpha,2}$, the Anderson hamiltonian $\mathscr{H}u $ can be written as
\begin{align}
    \mathscr{H}u  = &  \Phi(u) + \mathscr{L}u^{\sharp} - u\diamond\xi \nonumber\\
                 = & \mathscr{L}u^{\sharp} - R(u)\circ\xi -u^{\sharp}\circ\xi - u \prec \mathscr{U}^{L,\gamma_{1}}_{\leq} \xi - u\succ \mathscr{U}^{L,\gamma_{1}}_{\leq} \xi - u\succ \mathscr{U}^{K,\gamma_{2}}_{\leq}(\vartheta\circ\xi)  \nonumber\\
                 & - u \prec \mathscr{U}^{K,\gamma_{2}}_{\leq}(\vartheta\circ\xi)- C(u,\vartheta,\xi) -u\circ(\vartheta \diamond \xi) \nonumber \\
                 := & \mathscr{L}u^{\sharp} - \Psi(u).
\end{align}
Using the ansatz $u = u\prec\vartheta + R(u)  + u^{\sharp}$, we decompose the original equation (\ref{NAME}) into a simpler elliptic system
\begin{equation}\label{decPAM}
  \left\{
   \begin{aligned}
   & \mathscr{L}(u\prec\vartheta + R(u)) = \Phi(u) + \lambda (u\prec\vartheta + R(u)), \\
   & \mathscr{L} u^{\sharp} = a|u|^2 u + \Psi(u) + \lambda u^{\sharp},
   \end{aligned}
   \right.
\end{equation}

The renormalization argument (Lemma \ref{renormalize}) implies that $(\xi_{\epsilon},\vartheta_{\epsilon}\diamond\xi_{\epsilon}) \rightarrow (\xi, \vartheta\diamond\xi)$ in $\mathscr{C}^{-1-\kappa}\times\mathscr{C}^{-2\kappa}$ as $\epsilon \rightarrow 0$. In order to approximate the Anderson hamiltonian $\mathscr{H}$, we first define
\begin{align*}
    \Phi_{\epsilon}(u) := &  u \prec \mathscr{U}^{L,\gamma_{1}}_{>} \xi_{\epsilon} + u \succ \mathscr{U}^{L,\gamma_{1}}_{>} \xi_{\epsilon} + u \succ \mathscr{U}^{K,\gamma_{2}}_{>}(\vartheta_{\epsilon}\diamond\xi_{\epsilon}) + u \prec\mathscr{U}^{K,\gamma_{2}}_{>}(\vartheta_{\epsilon}\diamond\xi_{\epsilon}), \\
    \Psi_{\epsilon}(u):= &  u^{\sharp}\circ\xi_{\epsilon} + R(u)\circ\xi_{\epsilon} + C(u,\vartheta_{\epsilon},\xi_{\epsilon})+ u\prec\mathscr{U}^{K,\gamma_{2}}_{\leq}(\vartheta_{\epsilon}\diamond\xi_{\epsilon}) + u\succ\mathscr{U}^{K,\gamma_{2}}_{\leq}(\vartheta_{\epsilon}\diamond\xi_{\epsilon}) + u \circ(\vartheta_{\epsilon}\diamond\xi_{\epsilon})\\
     & + u\prec \mathscr{U}^{L,\gamma_{1}}_{\leq} \xi_{\epsilon} + u\succ \mathscr{U}^{L,\gamma_{1}}_{\leq}\xi_{\epsilon} .
\end{align*}
Then the approximation operator $\mathscr{H}_{\epsilon}$ is defined as
\begin{align}
    \mathscr{H}_{\epsilon}u  = & \mathscr{L}u^{\sharp} - R_{\epsilon}(u)\circ\xi -u^{\sharp}\circ\xi_{\epsilon} - u \prec \mathscr{U}^{L,\gamma_{1}}_{\leq} \xi_{\epsilon} - u\succ \mathscr{U}^{L,\gamma_{1}}_{\leq} \xi_{\epsilon} - u\succ \mathscr{U}^{K,\gamma_{2}}_{\leq}(\vartheta_{\epsilon}\circ\xi_{\epsilon})  \nonumber\\
                 & - u \prec \mathscr{U}^{K,\gamma_{2}}_{\leq}(\vartheta_{\epsilon}\circ\xi_{\epsilon})- C(u,\vartheta_{\epsilon},\xi_{\epsilon}) -u\circ(\vartheta_{\epsilon} \diamond \xi_{\epsilon}) \nonumber\\
                 = & \mathscr{L}u^{\sharp} - \Psi_{\epsilon}(u),
\end{align}
where $R_{\epsilon}$ is the bounded approximation linear operator of $R$ which given by
\begin{equation}\label{Reps}
    R_{\epsilon}(u)= \mathscr{L}^{-1}( \Phi_{\epsilon}(u) - \mathscr{L}(u \prec \vartheta_{\epsilon})).
\end{equation}

In order to obtain the properties of the space $\mathscr{D}^{\alpha,\beta}_{\vartheta}$, the following estimates for $u\prec\vartheta + R(u)$ and $R(u)$ in this subsection are essential.

\begin{lem}\label{HE11}
For every $u \in L^2$, we have
\begin{equation}
    \|u\prec\vartheta + R(u)\|_{L^2} \leq \frac{1}{2}\|u\|_{L^2}.
\end{equation}
\end{lem}
\begin{proof}
Note that the term $ u\prec\vartheta + R(u)$ satisfies $\mathscr{L}(u\prec\vartheta + R(u)) = \Phi(u)$. We control $\Phi(u)$ by Localization operators with suitable parameters. By Lemma \ref{Localization}, we employ Localization operators $\mathscr{U}^{L,\gamma_{1}}_{\leq}$ and $\mathscr{U}^{L,\gamma_{1}}_{>}$ with the parameter $L$, and $\gamma_{1} = 1+2\kappa$ such that
\begin{equation*}
    \|\mathscr{U}^{L,\gamma_{1}}_{>} \xi\|_{\mathscr{C}^{ -2+\kappa }} \lesssim 2^{-(1-2\kappa )L}\|\xi\|_{\mathscr{C}^{-1-\kappa}}.
\end{equation*}
Then by Bony's paraproduct estimate, we have
\begin{align}
    & \|u \prec \mathscr{U}^{L,\gamma_{1}}_{>} \xi\|_{H^{-2}} + \|u \succ \mathscr{U}^{L,\gamma_{1}}_{>} \xi\|_{H^{-2}}\nonumber \\
   \lesssim   & \|\mathscr{U}^{L}_{> } \xi\|_{\mathscr{C}^{-2+\kappa }} \|u\|_{L^2}  \nonumber \\
     \lesssim & 2^{-(1-2\kappa)L}\|\xi\|_{\mathscr{C}^{-1-\kappa}} \|u\|_{L^2}.
\end{align}
Similarly, we employ the Localization operators $\mathscr{U}^{K,\gamma_{2}}_{\leq}$ and $\mathscr{U}^{K,\gamma_{2}}_{>}$ with parameters $K$ and $\gamma_{2} = 3\kappa$ such that
\begin{equation*}
    \|\mathscr{U}^{K,\gamma_{2}}_{>} (\vartheta\diamond\xi)\|_{\mathscr{C}^{ -2+\kappa }} \lesssim 2^{-(2-3\kappa)K}\|\vartheta\diamond\xi\|_{\mathscr{C}^{-2\kappa}}.
\end{equation*}
Then
\begin{equation}
    \|u \prec \mathscr{U}^{K,\gamma_{2}}_{>} (\vartheta\diamond\xi)\|_{H^{-2}} + \|u \succ \mathscr{U}^{K,\gamma_{2}}_{>} (\vartheta\diamond\xi) \|_{H^{-2}} \lesssim 2^{-(2-3\kappa )K}\|\vartheta\diamond\xi\|_{\mathscr{C}^{-2\kappa}}\|u\|_{L^2}.
\end{equation}
The stochastic terms $\xi$ and $\vartheta\circ\xi$ can be controlled via choosing suitable $L,K>1$, such that
\begin{align*}
     \| u\prec\vartheta + R(u)\|_{L^2} \lesssim & \|\Phi(u)\|_{H^{-2}} \\
    \lesssim & \|u \prec \mathscr{U}_{>} \xi + u \succ \mathscr{U}_{>} \xi+ u \prec\mathscr{U}_{>}(\vartheta\diamond\xi) + u \succ \mathscr{U}_{>}(\vartheta\diamond\xi)\|_{H^{-2}} \\
    \lesssim & (2^{-(1-2\kappa)L}\|\xi\|_{\mathscr{C}^{-1-\kappa}} + 2^{-(2-3\kappa)K}\|\vartheta\circ\xi\|_{\mathscr{C}^{-2\kappa}})\|u\|_{L^2} \\
    \leq &\frac{1}{2}\|u\|_{L^2}.
\end{align*}
The proof is completed.
\end{proof}

We turn to estimate $R(u)$. 

\begin{lem}\label{RE}
Let $\alpha \in (2/3,1)$. Then for every $u \in H^{\alpha}$, we have
\begin{equation}
    \|R(u)\|_{H^{2\alpha}} \lesssim (\|\xi\|^2_{\mathscr{C}^{-1-\kappa}} + \|\vartheta\diamond\xi\|_{\mathscr{C}^{-2\kappa}}) \|u\|_{H^{\alpha}}, \quad \|R_{\epsilon}(u)\|_{H^{2\alpha}} \lesssim (\|\xi_{\epsilon}\|^2_{\mathscr{C}^{-1-\kappa}} + \|\vartheta_{\epsilon}\diamond\xi_{\epsilon}\|_{\mathscr{C}^{-2\kappa}}) \|u\|_{H^{\alpha}},
\end{equation}
and
\begin{equation}
    \|R(u)\|_{H^{\alpha}} \lesssim (\|\xi\|_{\mathscr{C}^{-1-\kappa}} + \|\vartheta\diamond\xi\|_{\mathscr{C}^{-2\kappa}}) \|u\|_{L^2}, \quad \|R_{\epsilon}(u)\|_{H^{\alpha}} \lesssim (\|\xi_{\epsilon}\|_{\mathscr{C}^{-1-\kappa}} + \|\vartheta_{\epsilon}\diamond\xi_{\epsilon}\|_{\mathscr{C}^{-2\kappa}}) \|u\|_{L^2},
\end{equation}
Furthermore, $R$ can be approximated by $R_{\epsilon}$ in the following sense
\begin{equation}
    \lim_{\epsilon\rightarrow 0}\|R - R_{\epsilon}\|_{L(H^{2\alpha }, H^{\alpha})} = 0.
\end{equation}
\end{lem}
\begin{proof}
Note that 
\begin{align*}
       \mathscr{L} R(u)  = & - u \prec \mathscr{U}^{L,\gamma_{1}}_{\leq} \xi + u \succ \mathscr{U}^{L,\gamma_{1}}_{>} \xi + u \succ \mathscr{U}^{K,\gamma_{2}}_{>}(\vartheta\diamond\xi) + u \prec\mathscr{U}^{K,\gamma_{2}}_{>}(\vartheta\diamond\xi)\\
        & - (\mathscr{L}u) \prec \vartheta + \nabla u \prec \nabla\vartheta.
\end{align*}
By paraproduct estimates, we have
\begin{align}\label{ER1}
    \|(\mathscr{L}u)\prec\vartheta\|_{H^{2\alpha-2}} & \lesssim \|\mathscr{L}u\|_{H^{\alpha-2}}\|\vartheta\|_{\mathscr{C}^{1-\kappa}} \nonumber\\
    & \lesssim \|\xi\|_{\mathscr{C}^{-1-\kappa}}\|u\|_{L^2},
\end{align}
and
\begin{align}
    \|\nabla u \prec \nabla\vartheta\|_{H^{2\alpha-2}} & \lesssim \|\nabla u\|_{H^{\alpha -1}}\|\nabla\vartheta\|_{\mathscr{C}^{-\kappa}} \nonumber \\
    & \lesssim \|\xi\|_{\mathscr{C}^{-1-\kappa}}\|u\|_{L^2}.
\end{align}
By localization operators with parameters $L,K$, and paraproduct estimates, we get
\begin{align}\label{ER3}
     & \|u\succ \mathscr{U}^{L,\gamma_{1}}_{>} \xi + u \succ \mathscr{U}^{K,\gamma_{2}}_{>}(\vartheta\diamond\xi) + u\prec \mathscr{U}^{L,\gamma_{1}}_{\leq} \xi + u \prec \mathscr{U}^{K,\gamma_{2}}_{>}(\vartheta\diamond\xi) \|_{H^{2\alpha -2} } \nonumber \\
    \lesssim & (\| \mathscr{U}^{L,\gamma_{1}}_{>} \xi\|_{\mathscr{C}^{\alpha-2}} + \| \mathscr{U}^{K,\gamma_{2}}_{>} (\vartheta\diamond\xi)\|_{\mathscr{C}^{\alpha-2}})\|u\|_{H^{\alpha} } \\
    & +(\| \mathscr{U}^{L,\gamma_{1}}_{\leq} \xi\|_{\mathscr{C}^{2\alpha-\kappa-2}}+ \| \mathscr{U}^{K,\gamma_{2}}_{>} (\vartheta\diamond\xi)\|_{\mathscr{C}^{2\alpha-\kappa-2}})\|u\|_{L^2} \nonumber \\
   \lesssim  & (2^{-(1-\kappa-\alpha)L}\|\xi\|_{\mathscr{C}^{-1-\kappa}}+2^{-(2-2\kappa-\alpha)K}\|\vartheta\diamond\xi\|_{\mathscr{C}^{-2\kappa}})\|u\|_{H^{\alpha} } \nonumber \\
   & +  (2^{(2\alpha-1 )L}\|\xi\|_{\mathscr{C}^{-1-\kappa}}+2^{(\kappa -2+2\alpha)K}\|\vartheta\diamond\xi\|_{\mathscr{C}^{-2\kappa}})\|u\|_{L^2} \nonumber\\
   \lesssim & (\|\xi\|^2_{\mathscr{C}^{-1-\kappa}} + \|\vartheta\diamond\xi\|_{\mathscr{C}^{-2\kappa}})\|u\|_{H^{\alpha}}.
\end{align}
Combining with above estimates (\ref{ER1})-(\ref{ER3}), we estimate
\begin{align*}
    \|R(u)\|_{H^{2\alpha}} \lesssim & \|\mathscr{L}R(u)\|_{H^{2\alpha-2}} \\
    \lesssim & \|\mathscr{L}u\prec\vartheta\|_{H^{2\alpha-2}}+\|\nabla u \prec \nabla\vartheta\|_{H^{2\alpha-2}}+ \|u\succ \mathscr{U}_{>} \xi + u \succ \mathscr{U}_{>}(\vartheta\diamond\xi)\|_{H^{2\alpha -2} } \\
     & +\|u\prec \mathscr{U}_{\leq} \xi  - u \prec \mathscr{U}_{>}(\vartheta\diamond\xi)\|_{H^{2\alpha -2} } \\
    \lesssim & (\|\xi\|^2_{\mathscr{C}^{-1-\kappa}} + \|\vartheta\diamond\xi\|_{\mathscr{C}^{-2\kappa}})\|u\|_{H^{\alpha}}.
\end{align*}
By same argument, we obtain same estimate for $R_{\epsilon}$. 

Similarly, we estimate
\begin{align}\label{ER12}
    \|(\mathscr{L}u)\prec\vartheta\|_{H^{\alpha-2}} \lesssim \|\xi\|_{\mathscr{C}^{-1-\kappa}}\|u\|_{L^2},
\end{align}
\begin{align}
    \|\nabla u \prec \nabla\vartheta\|_{H^{\alpha-2}} \lesssim \|\xi\|_{\mathscr{C}^{-1-\kappa}}\|u\|_{L^2},
\end{align}
and
\begin{align}\label{ER32}
     & \|u\succ \mathscr{U}^{L,\gamma_{1}}_{>} \xi + u \succ \mathscr{U}^{K,\gamma_{2}}_{>}(\vartheta\diamond\xi) + u\prec \mathscr{U}^{L,\gamma_{1}}_{\leq} \xi + u \prec \mathscr{U}^{K,\gamma_{2}}_{>}(\vartheta\diamond\xi) \|_{H^{\alpha -2} } \nonumber \\
   \lesssim  & (2^{-(1-\kappa-\alpha)L}\|\xi\|_{\mathscr{C}^{-1-\kappa}}+2^{-(2-2\kappa-\alpha)K}\|\vartheta\diamond\xi\|_{\mathscr{C}^{-2\kappa}})\|u\|_{L^2 } \nonumber \\
   & +  (2^{(\alpha-1 )L}\|\xi\|_{\mathscr{C}^{-1-\kappa}}+2^{(\kappa -2+\alpha)K}\|\vartheta\diamond\xi\|_{\mathscr{C}^{-2\kappa}})\|u\|_{L^2} \nonumber\\
   \lesssim & (\|\xi\|_{\mathscr{C}^{-1-\kappa}} + \|\vartheta\diamond\xi\|_{\mathscr{C}^{-2\kappa}})\|u\|_{L^2}.
\end{align}
Combining with above estimates (\ref{ER12})-(\ref{ER32}), it follows that
\begin{equation}
    \|R(u)\|_{H^{2\alpha}} \lesssim (\|\xi\|_{\mathscr{C}^{-1-\kappa}} + \|\vartheta\diamond\xi\|_{\mathscr{C}^{-2\kappa}})\|u\|_{L^2}.
\end{equation}
The estimate for $R_{\epsilon}$ is obtained via same argument.

Now we turn to prove the approximation result. For every $u \in H^{\alpha}$, we have
\begin{align*}
    & R(u) - R_{\epsilon}(u) \\
    = &  u \prec \mathscr{U}^{L,\gamma_{1}}_{\leq} (\xi-\xi_{\epsilon}) - u \succ \mathscr{U}^{L,\gamma_{1}}_{>} (\xi-\xi_{\epsilon}) - u \succ \mathscr{U}^{K,\gamma_{2}}_{>}(\vartheta\diamond\xi -\vartheta_{\epsilon}\diamond\xi_{\epsilon}) - u \prec\mathscr{U}^{K,\gamma_{2}}_{>}(\vartheta\diamond\xi -\vartheta_{\epsilon}\diamond\xi_{\epsilon}) \\
     & - (\mathscr{L}u) \prec (\vartheta-\vartheta_{\epsilon}) - \nabla u \prec \nabla (\vartheta-\vartheta_{\epsilon}).
\end{align*}
Since $(\xi_{\epsilon},\vartheta_{\epsilon}\diamond\xi_{\epsilon}) \rightarrow (\xi, \vartheta\diamond\xi)$ in $\mathscr{C}^{-1-\kappa}\times\mathscr{C}^{-2\kappa}$ as $\epsilon \rightarrow 0$, similar with estimate (\ref{ER1})-(\ref{ER3}), it follows that $\lim_{\epsilon\rightarrow 0}\|R - R_{\epsilon}\|_{L(H^{2\alpha }, H^{\alpha})} = 0.$ The proof is completed.
\end{proof}

We define the linear bounded operator $\Pi u = u^{\sharp} = u - u\prec\vartheta - R(u)$ for every $u \in \mathscr{D}_{\vartheta}^{\alpha,1}$. Similar with $R_{\epsilon}$, we also introduce the approximation operator $\Pi_{\epsilon}$ as $\Pi_{\epsilon}u = u- u\prec \vartheta_{\epsilon} - R_{\epsilon}(u)$.
The renormalization Lemma \ref{renormalize} implies that
$(\xi_{\epsilon},\vartheta_{\epsilon}\diamond\xi_{\epsilon}) \rightarrow (\xi, \vartheta\diamond\xi)$ in $\mathscr{C}^{-1-\kappa}\times\mathscr{C}^{-2\kappa}$ as $\epsilon \rightarrow 0$. 
By the choosing of $L,K$ in Lemma \ref{HE11}, we show that $\Pi$ is bijective in next Lemma.

\begin{lem}\label{Gam}
The bounded linear operators $\Pi$ and $\Pi_{\epsilon}$ are injective. The inverse operators $\Gamma$ and $\Gamma_{\epsilon}$ of  $\Pi$ and $\Pi_{\epsilon}$ satisfies
\begin{equation}\label{EGamma}
    \frac{1}{2}\|\Gamma u^{\sharp}\|_{L^2} \leq \|u^{\sharp}\|_{L^2}, \quad \frac{1}{2}\|\Gamma_{\epsilon} u^{\sharp}\|_{L^2} \leq \|u^{\sharp}\|_{L^2}.
\end{equation}
Furthermore, $\Gamma$ can be approximated by $\Gamma_{\epsilon}$ in the following sense
\begin{equation}\label{GAE}
    \lim_{\epsilon\rightarrow 0}\|\Gamma - \Gamma_{\epsilon}\|_{L(\mathscr{D}_{\vartheta}^{\alpha,1}, H^{1})} = 0.
\end{equation}
\end{lem}
\begin{proof}
By Lemma \ref{HE11}, for every $u_{1}, u_{2}\in \mathscr{D}_{\vartheta}^{\alpha,1}$, we have
\begin{align}
    \|u_{1} - u_{2}\|_{L^2} = & \|u_{1}\prec \vartheta + R(u_{1}) + u^{\sharp}_{1} - u_{1}\prec \vartheta - R(u_{2}) - u^{\sharp}_{2} \|_{L^2} \nonumber\\
     \leq & \|u_{1}\prec \vartheta + R(u_{1}) - u_{1}\prec \vartheta - R(u_{2}) \|_{L^2} + \|u^{\sharp}_{1} - u^{\sharp}_{2} \|_{L^2} \nonumber \\
     \leq & \frac{1}{2}\|u_{1} - u_{2}\|_{L^2} + \|u^{\sharp}_{1} - u^{\sharp}_{2} \|_{L^2}.
\end{align}
Thus
\begin{equation}
    \frac{1}{2}\|u_{1} - u_{2}\|_{L^2} \leq \|u^{\sharp}_{1} - u^{\sharp}_{2} \|_{L^2} = \| \Pi u_{1} - \Pi u_{2} \|_{L^2}.
\end{equation}
It implies that $\Pi$ is injective. Moreover, the operator $\Pi$ has the inverse operator $\Gamma$ such that for each $u \in \mathscr{D}_{\vartheta}^{\alpha,1}$,
\begin{equation}
    u = \Gamma u^{\sharp} = (\Gamma u^{\sharp})\prec\vartheta + R(\Gamma u^{\sharp} ) + u^{\sharp}.
\end{equation}
Moreover, it follows that
\begin{equation*}
    \|u\|_{L^{2}} = \|\Gamma u^{\sharp}\|_{L^2} \leq \frac{1}{2}\|u\|_{L^2}+\|u^{\sharp}\|_{L^2}.
\end{equation*}
Thus 
\begin{equation}
    \frac{1}{2}\|\Gamma u^{\sharp}\|_{L^2} \leq \|u^{\sharp}\|_{L^2}.
\end{equation}
The same argument show that the approximated operator $\Pi_{\epsilon}$ is also injective with inverse operator $\Gamma_{\epsilon}$ which satisfies that for each $u \in \mathscr{D}_{\vartheta}^{\alpha,1}$,
\begin{equation}
    u = \Gamma_{\epsilon} u^{\sharp} =  (\Gamma_{\epsilon} u^{\sharp})\prec\vartheta_{\epsilon} + R_{\epsilon}(\Gamma_{\epsilon} u^{\sharp}) + u^{\sharp},
\end{equation}
and 
\begin{equation}
    \frac{1}{2}\|\Gamma_{\epsilon} u^{\sharp}\|_{L^2} \leq \|u^{\sharp}\|_{L^2}.
\end{equation}
Now we turn to prove the approximation result. By Lemma \ref{HE11} and Lemma \ref{RE}, for each $u \in \mathscr{D}_{\vartheta}^{\alpha,1}$ we estimate
\begin{align*}
     & \|\Gamma u^{\sharp}- \Gamma_{\epsilon} u^{\sharp}\|_{L^2} \\
    = & \|(\Gamma u^{\sharp})\prec\vartheta + R(\Gamma u^{\sharp}) - (\Gamma_{\epsilon} u^{\sharp})\prec\vartheta_{\epsilon} - R_{\epsilon}(\Gamma_{\epsilon} u^{\sharp})\|_{L^2} \\
    \leq & \|(\Gamma u^{\sharp}- \Gamma_{\epsilon} u^{\sharp})\prec\vartheta + R(\Gamma u^{\sharp}- \Gamma_{\epsilon} u^{\sharp})\|_{L^2}+\|(\Gamma_{\epsilon} u^{\sharp})\prec(\vartheta -\vartheta_{\epsilon}) \|_{L^2}+ \| R(\Gamma_{\epsilon} u^{\sharp})  - R_{\epsilon}(\Gamma_{\epsilon} u^{\sharp})\|_{L^2} \\
    \leq & \frac{1}{2}\|\Gamma u^{\sharp}- \Gamma_{\epsilon} u^{\sharp}\|_{L^2} + C'(\|\vartheta_{\epsilon} - \vartheta\|_{\mathscr{C}^{1-\kappa}} +\|\xi_{\epsilon}- \xi \|_{\mathscr{C}^{-1-\kappa}}+ \|\vartheta_{\epsilon}\diamond\xi_{\epsilon}- \vartheta\diamond\xi \|_{\mathscr{C}^{-2\kappa}})\|\Gamma_{\epsilon} u^{\sharp}\|_{L^2},
\end{align*}
for some constant $C'>0$. Since $\lim_{\epsilon\rightarrow 0}\|\xi_{\epsilon}- \xi \|_{\mathscr{C}^{-1-\kappa}} = 0$, $\lim_{\epsilon\rightarrow 0}\|\vartheta_{\epsilon}\diamond\xi_{\epsilon}- \vartheta\diamond\xi \|_{\mathscr{C}^{-2\kappa}} = 0$, we have
\begin{equation}
    \lim_{\epsilon\rightarrow 0}\|\Gamma u^{\sharp}- \Gamma_{\epsilon} u^{\sharp}\|_{L^2} = 0.
\end{equation}
Then by Lemma \ref{RE}, it follows that
\begin{align*}
     & \|\Gamma u^{\sharp}- \Gamma_{\epsilon} u^{\sharp}\|_{\mathscr{D}_{\vartheta}^{\alpha,1}} \\
    = & \|(\Gamma u^{\sharp})\prec\vartheta + R(\Gamma u^{\sharp}) - (\Gamma_{\epsilon} u^{\sharp})\prec\vartheta_{\epsilon} - R_{\epsilon}(\Gamma_{\epsilon} u^{\sharp})\|_{H^{\alpha}} \\
    \leq & \|(\Gamma u^{\sharp}- \Gamma_{\epsilon} u^{\sharp})\prec\vartheta \|_{H^{\alpha}} + \| R(\Gamma u^{\sharp}- \Gamma_{\epsilon} u^{\sharp})\|_{H^{\alpha}}+\|(\Gamma_{\epsilon} u^{\sharp})\prec(\vartheta -\vartheta_{\epsilon}) \|_{H^{\alpha}}+ \| R(\Gamma_{\epsilon} u^{\sharp})  - R_{\epsilon}(\Gamma_{\epsilon} u^{\sharp})\|_{H^{\alpha}} \\
    \lesssim & \|\Gamma u^{\sharp}- \Gamma_{\epsilon} u^{\sharp}\|_{L^2}\|\vartheta\|_{\mathscr{C}^{1-\kappa}} + C'(\|\vartheta_{\epsilon} - \vartheta\|_{\mathscr{C}^{1-\kappa}} + \|\xi_{\epsilon}- \xi \|_{\mathscr{C}^{-1-\kappa}}+ \|\vartheta_{\epsilon}\diamond\xi_{\epsilon}- \vartheta\diamond\xi \|_{\mathscr{C}^{-2\kappa}})\|\Gamma_{\epsilon} u^{\sharp}\|_{H^{\alpha}},
\end{align*}
for some constant $C'>0$. Since $\lim_{\epsilon\rightarrow 0}\|\xi_{\epsilon}- \xi \|_{\mathscr{C}^{-1-\kappa}} = 0$, $\lim_{\epsilon\rightarrow 0}\|\vartheta_{\epsilon}\diamond\xi_{\epsilon}- \vartheta\diamond\xi \|_{\mathscr{C}^{-2\kappa}} = 0$, the map $\Gamma$ can be approximated by $\Gamma_{\epsilon}$ as
\begin{equation}
    \lim_{\epsilon\rightarrow 0}\|\Gamma - \Gamma_{\epsilon}\|_{L(\mathscr{D}_{\vartheta}^{\alpha,1}, H^{1})} = 0.
\end{equation}
This completes the proof.

\end{proof}

Using linear operators $\Gamma$ and $\Gamma_{\epsilon}$, we have the following result for space $\mathscr{D}_{\vartheta}^{\alpha,\beta}$.
\begin{thm}\label{Hr35}
Let $ 0 \leq \alpha < 1$. The space $\mathscr{D}_{\vartheta}^{\alpha,\beta}$ $(\beta\in \{1,2\})$ is a Hilbert space. Moreover, $ \mathscr{D}_{\vartheta}^{\alpha,\beta}$ is dense in $H^{\alpha}$.
\end{thm}

\begin{proof}
Since $\langle \cdot, \cdot \rangle_{ \mathscr{D}_{\vartheta}^{\alpha,\beta}}$ is an inner product, in order to prove that the space $\mathscr{D}_{\vartheta}^{\alpha,\beta}$ is a Hilbert space, it remains to show that the space $\mathscr{D}_{\vartheta}^{\alpha,\beta}$ is complete. Assume that $(u_n)_{n\geq 1}$ is a Cauchy sequence in $ \mathscr{D}_{\vartheta}^{\alpha,\beta}$. Then there exists $u \in H^{\alpha} $, $u^{\sharp} \in H^\beta$ such that
\begin{equation*}
     \lim_{n\rightarrow \infty}\|u_n - u\|_{H^{\alpha}}=0, \quad \lim_{n\rightarrow \infty}\|u^{\sharp}_{n}- u^{\sharp}\|_{H^{\beta}}=0.
\end{equation*}
By $(\ref{GAE})$, it follows that
\begin{equation*}
    \lim_{n\rightarrow \infty}\|u_n\prec\vartheta +R(u_n)- u\prec\vartheta  - R(u)\|_{H^{\alpha}}=0
\end{equation*}
It implies that $u^{\sharp} = u- u\prec\vartheta - R(u)$, and the limit $u: = \lim_{n \to \infty} u_{n} \in \mathscr{D}_{\vartheta}^{\alpha,\beta} $. Thus the space $\mathscr{D}_{\vartheta}^{\alpha,\beta}$ is complete.

Now we turn to prove that $\mathscr{D}_{\vartheta}^{\alpha,\beta}$ is dense in $H^{\alpha}$. For every $f \in H^{\alpha}$, by the paraproduct estimates and Lemma \ref{RE}, we have $f\prec \vartheta, R(f) \in H^{\alpha}$.
Since $H^{\beta}$ is dense in $H^{\alpha}$, we can choose a sequence $\{f^{\sharp}_{n}\} \subset H^{\beta}$ so that
\begin{equation}
    f^{\sharp}_{n} \rightarrow f-f\prec \vartheta- R(f)\text{ in } H^{\alpha} \text{ as } n\rightarrow \infty.
\end{equation}
Moreover, using the operator $\Gamma$ we can define $f_n := \Gamma f^{\sharp}_{n}) \in \mathscr{D}_{\vartheta}^{\alpha,\beta}$ so that
\begin{equation}
    f_{n} \rightarrow f \text{ in } H^{\alpha} \text{ as } n\rightarrow \infty.
\end{equation}
The proof is completed.
\end{proof}

\subsection{The energy functional and its Fr\'{e}chet derivative}

In this subsection, we define the energy functional $E(u)$ by the quadratic form of the Anderson hamiltonian $\mathscr{H}$, and derive the Fr\'{e}chet derivative of the energy functional of $E(u)$ on its domain $\mathscr{D}_{\vartheta}^{\alpha,1}$.

For every $v,u \in \mathscr{D}_{\vartheta}^{\alpha,2}$, we define the bilinear form $B_{\mathscr{H}}(v,u) = \langle v, \mathscr{H}u \rangle$, and estimate $B_{\mathscr{H}}(v,u)$ with norm $\|\cdot\|_{\mathscr{D}_{\vartheta}^{\alpha,1}}$ as follows.

\begin{thm}\label{ThmBh}
For every $v, u\in \mathscr{D}_{\vartheta}^{\alpha,2}$, we have
\begin{equation}\label{Tbh1}
    B_{\mathscr{H}}(v,u) \lesssim \|v\|_{\mathscr{D}_{\vartheta}^{\alpha,1}}\|u\|_{\mathscr{D}_{\vartheta}^{\alpha,1}}.
\end{equation}
and
\begin{equation}\label{Tbh2}
\frac{1}{2}\|u\|^2_{\mathscr{D}_{\vartheta}^{\alpha,1}}- C_{\xi} \|u\|^2_{L^2} \leq B_{\mathscr{H}}(u,u),
\end{equation}
where
\begin{equation}
    C_{\xi} = C(1+\|\xi\|_{\mathscr{C}^{-1-\kappa}} + \|\vartheta\diamond\xi\|_{\mathscr{C}^{-2\kappa}})^{4}.
\end{equation}
\end{thm}
\begin{proof}
For every $v, u\in \mathscr{D}_{\vartheta}^{\alpha,2}$, by integration by part we have
\begin{align}\label{BH1}
    \langle v, \mathscr{H}u \rangle = &  \langle v, \mathscr{L} u^{\sharp} \rangle -  \langle v , \Psi(u) \rangle \nonumber \\
     = & \langle  v^{\sharp}, \mathscr{L} u^{\sharp} \rangle + \langle  v\prec\vartheta + R(v), \mathscr{L} u^{\sharp} \rangle - D(v,\xi,u^{\sharp}) - \langle v\prec\xi, u^{\sharp} \rangle - \langle v , \Psi(u) - u^{\sharp}\circ \xi \rangle   \nonumber \\
    = & \langle \nabla v^{\sharp}, \nabla u^{\sharp} \rangle + \langle v^{\sharp},  u^{\sharp} \rangle + \langle \mathscr{L} R(v), u^{\sharp}\rangle  -  D(v,\xi,u^{\sharp}) \nonumber \\
     & + \langle (\mathscr{L}(v\prec\vartheta) - v \prec \xi), u^{\sharp} \rangle  - \langle v , \Psi(u) - u^{\sharp}\circ \xi \rangle,
\end{align}
where
\begin{equation}
    D(v,\xi,u^{\sharp}) = \langle v, u^{\sharp}\circ\xi \rangle - \langle v\prec\xi, u^{\sharp}\rangle.
\end{equation}
Lemma \ref{RE} implies that
\begin{equation}\label{451}
    \langle \mathscr{L} R(v), u^{\sharp}\rangle  \lesssim \|R(v)\|_{H^{2\alpha-2}}\|u^{\sharp}\|_{H^{\alpha}} \lesssim (\|\xi\|^2_{\mathscr{C}^{-1-\kappa}} + \|\vartheta\diamond\xi\|_{\mathscr{C}^{-2\kappa}}) \|v\|_{H^{\alpha}}\|u^{\sharp}\|_{H^{\alpha}}.
\end{equation}
By Lemma \ref{Dm}, we obtain
\begin{align}\label{452}
    D(v,\xi,u^{\sharp})\lesssim  \|\xi\|_{\mathscr{C}^{-1-\kappa}}\|v\|_{H^{\alpha}}\|u^{\sharp}\|_{H^{\alpha}}.
\end{align}
Using Lemma \ref{CE}, we have
\begin{equation}\label{453}
    \langle (\mathscr{L}(v\prec\vartheta) - v \prec \xi), u^{\sharp} \rangle \leq \|\mathscr{L}(v\prec\vartheta) - v \prec \xi\|_{L^2}\|u^{\sharp}\|_{L^2} \lesssim \|\xi\|_{\mathscr{C}^{-1-\kappa}}\|v\|_{H^{\alpha}}\|u^{\sharp}\|_{H^{\alpha}}
\end{equation}
By paraproduct estimates, commutator estimate, and Lemma \ref{RE}, we estimate
\begin{align}\label{Epsi}
    &\|\Psi(u) - u^{\sharp} \circ\xi \|_{L^2} \nonumber \\
    \leq & \|R(u)\circ\xi\|_{L^2} + \|C(u,\vartheta,\xi)\|_{L^2} + \|u\circ (\vartheta\diamond\xi)\|_{L^2} \nonumber \\
    & + \|u \prec \mathscr{U}^{L,\gamma_{1}}_{\leq} \xi + u\succ \mathscr{U}^{L,\gamma_{1}}_{\leq} \xi\|_{L^2} + \|u\prec \mathscr{U}^{K,\gamma_{2}}_{\leq}(\vartheta\circ\xi) + u \succ \mathscr{U}^{K,\gamma_{2}}_{\leq}(\vartheta\circ\xi)\|_{L^2} \nonumber \\
    \lesssim & (\|\xi\|^2_{\mathscr{C}^{-1-\kappa}} + \|\xi\|_{\mathscr{C}^{-1-\kappa}}\|\vartheta\diamond\xi\|_{\mathscr{C}^{-2\kappa}})\|u\|_{H^{\alpha}}  + \|\xi\|^2_{\mathscr{C}^{-1-\kappa}}\|u\|_{ H^{\alpha}} \nonumber \\
    & + \|\vartheta\diamond\xi\|_{\mathscr{C}^{-2\kappa}}\|u\|_{H^{\alpha}} + 2^{ (\alpha - 1 + 2\kappa)L}\|\xi\|_{\mathscr{C}^{-1-\kappa}}\|u\|_{H^{\alpha}} + 2^{3\kappa K}\|\vartheta\diamond\xi\|_{\mathscr{C}^{-2\kappa}}\|u\|_{H^{\alpha}} \nonumber\\
    \lesssim & (1+\|\xi\|^2_{\mathscr{C}^{-1-\kappa}} + \|\vartheta\diamond\xi\|^2_{\mathscr{C}^{-2\kappa}}) \|u\|_{H^{\alpha}}.
\end{align}
Thus
\begin{equation}\label{454}
    \langle v , \Psi(u) - u^{\sharp}\circ \xi\rangle \leq \|v\|_{L^2}\|\Psi(u) - u^{\sharp}\circ \xi\|_{L^2} \lesssim  (1 + \|\xi\|^2_{\mathscr{C}^{-1-\kappa}} + \|\vartheta\diamond\xi\|^2_{\mathscr{C}^{-2\kappa}})\|v\|_{L^2}\|u\|_{H^{\alpha}}.
\end{equation}

Combining above estimates (\ref{451})-(\ref{454}), we conclude that
\begin{align}
    & \langle v, \mathscr{H}u \rangle - \langle \nabla v^{\sharp}, \nabla u^{\sharp} \rangle - \langle v^{\sharp},  u^{\sharp} \rangle \nonumber \\
    \lesssim & (1+\|\xi\|_{\mathscr{C}^{-1-\kappa}} + \|\vartheta\diamond\xi\|_{\mathscr{C}^{-2\kappa}})^2(\|v^{\sharp}\|_{H^{\alpha}}+ \|v\|_{H^{\alpha}})(\|u^{\sharp}\|_{H^{\alpha}}+ \|u\|_{H^{\alpha}}).
\end{align}
Thus
\begin{equation*}
    \langle v, \mathscr{H}u \rangle \lesssim \|v\|_{\mathscr{D}_{\vartheta}^{\alpha,1}}\|u\|_{\mathscr{D}_{\vartheta}^{\alpha,1}}.
\end{equation*}
Now we take $u=v$. Then by the operator $\Pi$ and interpolation inequality, it follows that
\begin{align*}
      & \langle u, \mathscr{H}u \rangle - \| \nabla u^{\sharp} \|^2_{L^2} - \|u\|^2_{H^{\alpha}} \\
      \lesssim & (1+\|\xi\|_{\mathscr{C}^{-1-\kappa}} + \|\vartheta\diamond\xi\|_{\mathscr{C}^{-2\kappa}})^2 (\|u^{\sharp}\|^2_{H^{\alpha}} + \| u\|^2_{H^{\alpha}} )\nonumber \\
     \lesssim & (1+\|\xi\|_{\mathscr{C}^{-1-\kappa}} + \|\vartheta\diamond\xi\|_{\mathscr{C}^{-2\kappa}})^2 \|u^{\sharp}\|^2_{H^{\alpha}} \nonumber \\
    \lesssim & \delta \|\nabla u^{\sharp}\|^2_{L^2} + C_{\delta}(1+\|\xi\|_{\mathscr{C}^{-1-\kappa}} + \|\vartheta\diamond\xi\|_{\mathscr{C}^{-2\kappa}})^{4}  \|u^{\sharp}\|^2_{L^2}) \nonumber\\
    \lesssim & \delta \|\nabla u^{\sharp}\|^2_{L^2} + C_{\delta}(1+\|\xi\|_{\mathscr{C}^{-1-\kappa}} + \|\vartheta\diamond\xi\|_{\mathscr{C}^{-2\kappa}})^{4}  \|u\|^2_{L^2}).
\end{align*}
Thus
\begin{equation}
    \| \nabla u^{\sharp} \|^2_{L^2} + \|u\|^2_{H^{\alpha}} -  \delta \|\nabla u^{\sharp}\|^2_{L^2} - C_{\delta}(1+\|\xi\|_{\mathscr{C}^{-1-\kappa}} + \|\vartheta\diamond\xi\|_{\mathscr{C}^{-2\kappa}})^{4}  \|u\|^2_{L^2}) \leq \langle u, \mathscr{H}u \rangle.
\end{equation}
Choosing $\delta$ small enough to absorb $\|\nabla u\|^{2}_{L^2}$ in right hand side, we obtain
\begin{equation}
\frac{1}{2}\|u\|^2_{\mathscr{D}_{\vartheta}^{\alpha,1}}- C_{\xi}\|u\|^2_{L^2} \leq \langle u, \mathscr{H}u \rangle.
\end{equation}
where 
\begin{equation}
    C_{\xi} = C(1+\|\xi\|_{\mathscr{C}^{-1-\kappa}} + \|\vartheta\diamond\xi\|_{\mathscr{C}^{-2\kappa}})^{4}.
\end{equation}
\end{proof}

By Theorem \ref{ThmBh}, we can define $B_{\mathscr{H}}(v,u)$ as the quadratic form given by the Anderson hamiltonian $\mathscr{H}$.
\begin{defn}\label{thmBh}
For every $v, u\in \mathscr{D}_{\vartheta}^{\alpha,1}$, we define the bilinear form $B_{\mathscr{H}}(v,u)$ as \begin{align}
    B_{\mathscr{H}}(v,u) = & \langle \nabla v^{\sharp}, \nabla u^{\sharp} \rangle + \langle v^{\sharp},  u^{\sharp} \rangle + \langle \mathscr{L} R(v), u^{\sharp}\rangle  -  D(v,\xi,u^{\sharp}) \nonumber \\
     & + \langle (\mathscr{L}(v\prec\vartheta) - v \prec \xi), u^{\sharp} \rangle  - \langle v , \Psi(u) - u^{\sharp}\circ\xi \rangle.
\end{align}
\end{defn}

When $v, u\in \mathscr{D}_{\vartheta}^{\alpha,2}$, by (\ref{BH1}) it follows that $B_{\mathscr{H}}(v,u) = \langle v, \mathscr{H}u\rangle$. The $\Gamma$ map implies that  $\mathscr{D}_{\vartheta}^{\alpha,2}$ is dense in $\mathscr{D}_{\vartheta}^{\alpha,1}$. Thus by a density argument, the estimates (\ref{Tbh1}) and (\ref{Tbh2}) still hold, and the quadratic form $ B_{\mathscr{H}}(v,u)$ is closed, symmetric, and semi-bounded from $\mathscr{D}_{\vartheta}^{\alpha,1}\times\mathscr{D}_{\vartheta}^{\alpha,1}$ to $\mathbb{R}$.

Now we state the self-adjointness of the operator $\mathscr{H}+ C_{\xi} I$. It also implies the self-adjointness of the Anderson Hanmitonian $\mathscr{H}$. We remark that  similar result also appear in other papers, see e.g. \cite[Theorem 1.1]{AC2015} or \cite[Lemma 2.29]{GUZ2020}. However, we use a different high-low frequency decomposition, and the domain $\mathscr{D}_{\vartheta}^{\alpha,\beta}$ is little different with previous papers. Thus we still provide a self-contained proof for completeness in Appendix B.

\begin{thm}\label{sfH}
The operator $\mathscr{H}+ C_{\xi} I$ is a positive self-adjoint operator from $\mathscr{D}_{\vartheta}^{\alpha,1}$ to $L^2$, where $C_{\xi}$ is the positive constant given in Theorem \ref{ThmBh}.
\end{thm}

Now we define the energy functional
\begin{equation}\label{EFu}
    E(u) = \frac{1}{2}B_{\mathscr{H}}(u,u) + \frac{a}{4}\int_{\mathbb{T}^2} |u(x)|^4 dx, \quad u \in \mathscr{D}_{\vartheta}^{\alpha,1}.
\end{equation}

Recall that the energy functional $E(u)$ is Fr\'{e}chet differentiable at $u \in \mathscr{D}_{\vartheta}^{\alpha,1}$ if there exists a continuous linear functional $\mathscr{D}E)(u): \mathscr{D}_{\vartheta}^{\alpha,1}\rightarrow R$ such that for any $\epsilon>0$, there exists a $\delta = \delta(\epsilon,u)>0$ so that
\begin{equation*}
    |E(u+v)-E(u)-\mathscr{D}E)(u)(v)| < \epsilon\|v\|_{\mathscr{D}_{\vartheta}^{\alpha,1}} \quad \text{for every } \|v\|_{\mathscr{D}_{\vartheta}^{\alpha,1}} <\delta .
\end{equation*}
We call $\mathscr{D}E)(u)$ is the Fr\'{e}chet derivative of the energy functional $E(u)$.

Now we consider the Fr\'{e}chet derivative of the energy functional $E(u)$.

\begin{thm}\label{FDEA}
The energy functional $E(u)$ is well defined and of class $C^1$ on $\mathscr{D}_{\vartheta}^{\alpha,1}$. Its Fr\'{e}chet derivative at point $u$ is given by
\begin{equation}\label{DE}
    \mathscr{D}E(u)(v)= B_{\mathscr{H}}(v,u) + a\int_{\mathbb{T}^2} v|u|^2 u dx  \quad \text{for every } v \in \mathscr{D}_{\vartheta}^{\alpha,1}.
\end{equation}
\end{thm}

\begin{proof}
By Sobolev embedding theorem, for every $p\geq 1$ we can choose $\alpha $ large enough such that $\mathscr{D}_{\vartheta}^{\alpha,1} \hookrightarrow L^p$. Thus the energy functional $E(u)$ is meaningful and continuous on $\mathscr{D}_{\vartheta}^{\alpha,1}$.

Now we show that $E$ is differentiable in the sense of Fr\'{e}chet and that its Fr\'{e}chet derivative is given by (\ref{DE}). Since $\mathscr{D}_{\vartheta}^{\alpha,1} \hookrightarrow L^{p}$, the nonlinear term in (\ref{EFu}) is well-studied in Calculus of Variations (see e.g. \cite[Corollary 1.1.7]{C2005}). The nonlinear term is $C^1$ from $\mathscr{D}_{\vartheta}^{\alpha,1}$ to $\mathbb{R}$, and its Fr\'{e}chet derivative at $u \in \mathscr{D}_{\vartheta}^{\alpha,1}$ is
\begin{equation}\label{DE1}
    \mathscr{D}(\frac{a}{4}\int_{\mathbb{T}^2} |u(x)|^4 dx)(v) = a\int_{\mathbb{T}^2} v(x)|u(x)|^2 u(x) dx, \quad v \in \mathscr{D}_{\vartheta}^{\alpha,1}.
\end{equation}
It is enough to show that the map $u \rightarrow B_{\mathscr{H}}(u,u)$ is differentiable. For all $u,v\in \mathscr{D}_{\vartheta}^{\alpha,1}$, by Theorem \ref{ThmBh} it holds that
\begin{equation*}
      | B_{\mathscr{H}}(u+v, u+v)  - B_{\mathscr{H}}(u, u) - 2B_{\mathscr{H}}(v, u)|
    =  |B_{\mathscr{H}}(v, v) |
    \lesssim \|v\|^2_{\mathscr{D}_{\vartheta}^{\alpha,1}}.
\end{equation*}
Thus
\begin{equation}\label{DE2}
   | B_{\mathscr{H}}(u+v, u+v)  - B_{\mathscr{H}}(u, u) - 2B_{\mathscr{H}}(v, u)| =0 \quad \text{as } \|v\|_{\mathscr{D}_{\vartheta}^{\alpha,1}}\rightarrow 0.
\end{equation}
Combining with (\ref{DE1}) and (\ref{DE2}), we obtain that
\begin{equation*}
     \mathscr{D}E(u)(v)=B_{\mathscr{H}}(v,u) + a\int_{\mathbb{T}^2} v(x)|u(x)|^2 u dx, \quad v \in \mathscr{D}_{\vartheta}^{\alpha,1}.
\end{equation*}
Moreover, the Fr\'{e}chet derivative $u\rightarrow \mathscr{D}E(u)$ is continuous from $\mathscr{D}_{\vartheta}^{\alpha,1}$ to its dual $(\mathscr{D}_{\vartheta}^{\alpha,1})^{\ast} = \mathscr{D}_{\vartheta}^{\alpha,1}$. This finishes the proof of theorem.
\end{proof}

By the Fr\'{e}chet derivative of $E(u)$, we define the weak solution of singular stochastic partial differential equation (\ref{NAME}) as follows.
\begin{defn}
We say $u \in \mathscr{D}_{\vartheta}^{\alpha,1}$ is a weak solution of singular stochastic partial differential equation (\ref{NAME}) if
\begin{equation*}
     B_{\mathscr{H}}(v,u) + a\int_{\mathbb{T}^2} v(x)|u(x)|^2u(x) dx = \lambda \langle v,u \rangle, \quad v \in \mathscr{D}_{\vartheta}^{\alpha,1}.
\end{equation*}
\end{defn}

\subsection{Existence of minimizer}\label{PfEx}

In this subsection, we show the existence of minimizer by direct method in the calculus of variations.

\begin{proof}(Proof of Theorem \ref{Ex})
By Theorem \ref{ThmBh}, the energy functional $E(u)$ satisfies
\begin{align}
    E(u) = & B_{\mathscr{H}}(u,u) + \frac{a}{4}\int_{\mathbb{T}^2} |u(x)|^4 dx \nonumber \\
        \geq & c\|u\|^2_{\mathscr{D}_{\vartheta}^{\alpha,1}} -C_{\xi} \|u\|^2_{L^2} + \frac{a}{4}\|u\|^4_{L^4}  \nonumber\\
        \geq & c(\|u\|^2_{H^{\alpha}}+\|u^{\sharp}\|^2_{H^1})-C_{\xi}  \|u\|^2_{L^2} + \frac{a}{4}\|u\|^4_{L^4} .
\end{align}
Then we choose $\alpha$ large enough so that $H^{\alpha-\delta} \hookrightarrow L^{4}$ for some small $\delta>0$. Thus by interpolation inequality, we have
\begin{equation}
    E(u) \geq c \|u^{\sharp}\|^2_{H^1} - C'
\end{equation}
for some constant $C'>0$. Thus the energy functional $E(u)$ is bounded from below and coercive on $\mathscr{D}_{\vartheta}^{\alpha,1}$.

Now we turn to prove the weak lower-semicontinuity of $E(u)$. By Theorem \ref{ThmBh}, the map $u \rightarrow B_{\mathscr{H}}(u,u)+ C_{\xi}\|u\|^2_{L^2}$ is a convex functional from $\mathscr{D}_{\vartheta}^{\alpha,1}$ to $\mathbb{R}$. So for every $u_n, u \in \mathscr{D}_{\vartheta}^{\alpha,1}$, the energy functional $E(u_n)$ satisfies
\begin{align*}
    E(u_n) = & \frac{1}{2}B_{\mathscr{H}}(u_n, u_n) + \frac{a}{4}\int_{\mathbb{T}^2}|u_k(x)|^4 dx \\
           \geq & \frac{1}{2}(B_{\mathscr{H}}(u,u) + \lambda\|u\|^2_{L^2}) + B_{\mathscr{H}}(u_n-u,u) \\
           & \quad +\lambda \langle u_n -u, u \rangle - \lambda\|u_n\|^2_{L^2} + \frac{a}{4}\|u_{k}\|^4_{L^4}.
\end{align*}
For any sequence $\{u_n\}^{\infty}_{n=1}$ with $u_n \rightharpoonup u$ weakly in $\mathscr{D}_{\vartheta}^{\alpha,1}$, we have
\begin{equation}\label{WLC1}
    \lim_{n\rightarrow \infty} B_{\mathscr{H}}(u_n-u,u) + \lambda \langle u_n -u, u \rangle = 0.
\end{equation}
Since the weakly convergent sequence $\{u_n\}^{\infty}_{n=1}$ is bounded in $\mathscr{D}_{\vartheta}^{\alpha,1}$, the Sobolev compact embedding theorem implies that $u_n \rightarrow u$ strongly in $L^4$ as $n\rightarrow\infty$.
Thus
\begin{equation}\label{WLC2}
    \lim_{n\rightarrow \infty} (- \lambda\|u_n\|^2_{L^2}+ \frac{a}{4}\|u_{k}\|^4_{L^4}) = - \lambda\|u\|^2_{L^2} + \frac{a}{4}\|u_{k}\|^4_{L^4}.
\end{equation}
From (\ref{WLC1}) and (\ref{WLC2}), we get
\begin{equation}
    E(u) \leq \liminf_{n \rightarrow \infty} E(u_n).
\end{equation}
It implies the weak lower-semicontinuity of $E(u)$

Since $E(u)$ is bounded from below, there exists a minimizing sequence $\{u_n\}^{\infty}_{n=1} \subset \mathscr{D}_{\vartheta}^{\alpha,1}$ satisfying $\|u_{n}\|^2_{L^2} =1$ and
\begin{equation*}
    \lim_{n\rightarrow \infty}E(u_k) = \inf_{w \in \mathscr{D}_{\vartheta}^{\alpha,1}, \|w\|^2_{L^2}=1}E(w).
\end{equation*}
The coercivity of $E(u)$ implies that the minimizing sequence $\{u_n\}^{\infty}_{n=1}$ is bounded in $\mathscr{D}_{\vartheta}^{\alpha,1}$. Since $\mathscr{D}_{\vartheta}^{\alpha,1}$ is reflexive, there exists $u \in \mathscr{D}_{\vartheta}^{\alpha,1}$, such that $u_n \rightharpoonup u$ weakly in $\mathscr{D}_{\vartheta}^{\alpha,1}$ as $n\rightarrow \infty$.
Then the weak lower-semicontinuity implies that $u$ is a minimizer of $E(u)$. Moreover, by Theorem \ref{FDEA}, the minimizer $u$ is a weak solution of the Euler-Lagrange equation (\ref{NAME}). This completes the proof.
\end{proof}

\section{Regularity of the minimizer}\label{S4}

In this section, we study the regularity of the minimizer $u$. We establish $L^2$ estimates in Theorem \ref{L2E} and Schauder estimates for $u$.

\subsection{$L^2$ estimates}\label{sbL2R}

\begin{thm}\label{L2E}
The minimizer $u$ given in Theorem 1.1 satisfies $u\in\mathscr{D}_{\vartheta}^{\alpha,2}$.
\end{thm}
\begin{proof}
Let $u\in\mathscr{D}_{\vartheta}^{\alpha,1}$ be a weak solution for the elliptic singular stochastic partial differential equation (\ref{NAME}). Recall definition (\ref{thmBh}), for every $v\in \mathscr{D}_{\vartheta}^{\alpha,1}$, the remainder term $u^{\sharp}$ satisfies
\begin{align}\label{uvarW}
   \langle \nabla v^{\sharp}, \nabla u^{\sharp} \rangle   = & - \langle v^{\sharp},  u^{\sharp} \rangle - \langle \mathscr{L} R(v), u^{\sharp}\rangle + D(v,\xi,u^{\sharp}) - \langle (\mathscr{L}(v\prec\vartheta)- v \prec \xi), u^{\sharp} \rangle  \nonumber \\
     & + \langle v , G(u) \rangle + a\int_{\mathbb{T}^2} v(x)|u(x)|^2 u(x) dx + \lambda \langle v,u\rangle,
\end{align}
where
\begin{equation*}
    D(v,\xi,u^{\sharp}) = \langle v, u^{\sharp}\circ\xi \rangle - \langle v\prec\xi, u^{\sharp}\rangle.
\end{equation*}
Then we substitute $v^{\sharp} = 2^{2(1-\kappa)k}\Delta_k u^{\sharp}$, $v = \Gamma(2^{2(1-\kappa)k}\Delta_k u^{\sharp})$ into (\ref{uvarW}), and deduce
\begin{align}
     & \langle \nabla (2^{2(1-\kappa)k}\Delta_k u^{\sharp}), \nabla u^{\sharp} \rangle + \langle 2^{(1-\kappa)k}\Delta_k u^{\sharp},  2^{(1-\kappa)k}\Delta_k u^{\sharp} \rangle\nonumber \\
     = &  -\langle \mathscr{L} R(\Gamma(2^{2(1-\kappa)k}\Delta_k u^{\sharp})), u^{\sharp}\rangle + D(\Gamma(2^{2(1-\kappa)k}\Delta_k u^{\sharp}),\xi,u^{\sharp})   \nonumber \\
     & - \langle (\mathscr{L}(\Gamma(2^{2(1-\kappa)k}\Delta_k u^{\sharp}))\prec\vartheta)-\Gamma(2^{2(1-\kappa)k}\Delta_k u^{\sharp}) \prec \xi), u^{\sharp} \rangle +\langle \Gamma(2^{2(1-\kappa)k}\Delta_k u^{\sharp}) , \Psi(u) - u^{\sharp}\circ \xi \rangle \nonumber \\
     &+ \int_{\mathbb{T}^2} \Gamma(2^{2(1-\kappa)k}\Delta_k u^{\sharp})a|u|^2 u dx .
\end{align}
By Bernstein inequality Lemma \ref{BerI}, we have
\begin{equation}\label{R450}
    \|2^{2(2-\kappa)k}\Delta_k u^{\sharp}\|^2_{L^2} \lesssim \langle \nabla (2^{2(1-\kappa)k}\Delta_k u^{\sharp}), \nabla u^{\sharp} \rangle.
\end{equation}
Since $\|\Gamma(2^{2(1-\kappa)k}\Delta_k u^{\sharp})\|_{H^{\alpha}} \leq 2\| 2^{2(1-\kappa)k}\Delta_k u^{\sharp}\|_{H^{\alpha}} $, by Lemma \ref{Gam} and \ref{RE}, we have
\begin{align}\label{R451}
    \langle \mathscr{L} R(\Gamma(2^{2(1-\kappa)k}\Delta_k u^{\sharp})), u^{\sharp}\rangle  \lesssim & \|\mathscr{L} R(\Gamma(2^{2(1-\kappa)k}\Delta_k u^{\sharp}))\|_{H^{2\alpha -2}}\|u^{\sharp}\|_{H^{\alpha}} \nonumber \\
    \lesssim & \| R(\Gamma(2^{2(1-\kappa)k}\Delta_k u^{\sharp}))\|_{H^{2\alpha}}\|u^{\sharp}\|_{H^{\alpha}} \nonumber \\
    \lesssim & \|2^{2(1-\kappa)k}\Delta_k u^{\sharp}\|_{H^{\alpha}}\|u^{\sharp}\|_{H^{\alpha}}.
\end{align}
\begin{equation}\label{R452}
    D(\Gamma(2^{2(1-\kappa)k}\Delta_k u^{\sharp}),\xi,u^{\sharp})\lesssim  \|\xi\|_{\mathscr{C}^{-1-\kappa}}\|2^{2(1-\kappa)k}\Delta_k u^{\sharp}\|_{H^{\kappa}}\|u^{\sharp}\|_{H^{1}}.
\end{equation}
\begin{equation}\label{R453}
    \langle (\mathscr{L}(\Gamma(2^{2(1-\kappa)k}\Delta_k u^{\sharp})\prec\vartheta) - \Gamma(2^{2(1-\kappa)k}\Delta_k u^{\sharp}) \prec \xi), u^{\sharp} \rangle  \lesssim  \|\xi\|_{\mathscr{C}^{-1-\kappa}}\|2^{2(1-\kappa)k}\Delta_k u^{\sharp}\|_{H^{\kappa}}\|u^{\sharp}\|_{H^{1}}
\end{equation}
\begin{equation}\label{R454}
    \langle \Gamma(2^{2(1-\kappa)k}\Delta_k u^{\sharp}) , \Psi(u) - u^{\sharp}\circ \xi \rangle \leq \|2^{2(1-\kappa)k}\Delta_k u^{\sharp}\|_{L^2}\|\Psi(u) - u^{\sharp}\circ \xi\|_{L^2} \lesssim  \|2^{2(1-\kappa)k}\Delta_k u^{\sharp}\|_{L^2}\|u\|_{H^{\alpha}}.
\end{equation}
By Sobolev embedding, we obtain
\begin{equation}\label{R455}
    \int_{\mathbb{T}^2} \Gamma(2^{2(1-\kappa)k}\Delta_k u^{\sharp})a|u|^2 u dx \lesssim \|\Gamma(2^{2k}\Delta_k u^{\sharp})\|_{L^2}\|a|u|^2 u\|_{L^2} \lesssim \|2^{2(1-\kappa)k}\Delta_k u^{\sharp}\|_{L^2}\|u\|_{H^{\alpha}}.
\end{equation}
We combine above estimates (\ref{R450})-(\ref{R455}), and sum over $ k$ to get
\begin{align}
    \|u^{\sharp}\|^2_{H^{2-\kappa}} + \|u^{\sharp}\|^2_{H^{1-\kappa}} \lesssim \|u^{\sharp}\|_{H^{2-\kappa}}(\|u\|_{H^{\alpha}}+\|u^{\sharp}\|_{H^1}).
\end{align}
After using the weighted Young inequality and choosing $\delta$ small enough to absorb $\|u^{\sharp}\|_{H^{2-\kappa}}$ into the left hand side, we have
\begin{equation}
    \|u^{\sharp}\|^2_{H^{2-\kappa}} \lesssim \|u\|^2_{H^{\alpha}}+\|u^{\sharp}\|^2_{H^1}.
\end{equation}
Thus remainder term $u^{\sharp} \in H^{2-\kappa}$. Now we substitute $v^{\sharp} = 2^{2k}\Delta_k u^{\sharp}$, $v = \Gamma(2^{2k}\Delta_k u^{\sharp})$ into (\ref{uvarW}), and deduce
\begin{align}
     & \langle \nabla (2^{2k}\Delta_k u^{\sharp}), \nabla u^{\sharp} \rangle \nonumber \\
     = & -1\langle 2^{k}\Delta_k u^{\sharp},  2^{k}\Delta_k u^{\sharp} \rangle -\langle \mathscr{L} R(\Gamma(2^{2k}\Delta_k u^{\sharp})), u^{\sharp}\rangle + D(\Gamma(2^{2k}\Delta_k u^{\sharp}),\xi,u^{\sharp})   \nonumber \\
     & - \langle (\mathscr{L}(\Gamma(2^{2k}\Delta_k u^{\sharp}))\prec\vartheta)-\Gamma(2^{2k}\Delta_k u^{\sharp}) \prec \xi), u^{\sharp} \rangle +\langle \Gamma(2^{2k}\Delta_k u^{\sharp}) , \Psi(u) - u^{\sharp}\circ\xi \rangle \nonumber \\
     &+ \int_{\mathbb{T}^2} \Gamma(2^{2k}\Delta_k u^{\sharp})a|u|^2 u dx .
\end{align}
We estimate
\begin{equation}\label{R2450}
    \|2^{4k}\Delta_k u^{\sharp}\|^2_{L^2} \lesssim \langle \nabla (2^{2k}\Delta_k u^{\sharp}), \nabla u^{\sharp} \rangle.
\end{equation}
\begin{equation}\label{R2451}
    \langle \mathscr{L} R(\Gamma(2^{2k}\Delta_k u^{\sharp})), u^{\sharp}\rangle  \lesssim \|\mathscr{L} R(\Gamma(2^{2k}\Delta_k u^{\sharp}))\|_{H^{\alpha-2}}\|u^{\sharp}\|_{H^{2\alpha}} \lesssim \|2^{2k}\Delta_k u^{\sharp}\|_{L^2}\|u^{\sharp}\|_{H^{2\alpha}}.
\end{equation}
\begin{equation}\label{R2452}
    D(\Gamma(2^{2k}\Delta_k u^{\sharp}),\xi,u^{\sharp})\lesssim  \|\xi\|_{\mathscr{C}^{-1-\kappa}}\|2^{2k}\Delta_k u^{\sharp}\|_{L^2}\|u^{\sharp}\|_{H^{2-\kappa}}.
\end{equation}
\begin{equation}\label{R2453}
    \langle (\mathscr{L}(\Gamma(2^{2k}\Delta_k u^{\sharp})\prec\vartheta) - \Gamma(2^{2k}\Delta_k u^{\sharp}) \prec \xi), u^{\sharp} \rangle  \lesssim  \|\xi\|_{\mathscr{C}^{-1-\kappa}}\|2^{2k}\Delta_k u^{\sharp}\|_{L^2}\|u^{\sharp}\|_{H^{2-\kappa}}
\end{equation}
\begin{equation}\label{R2454}
    \langle \Gamma(2^{2k}\Delta_k u^{\sharp}) , \Psi(u) - u^{\sharp}\circ\xi \rangle \leq \|2^{2k}\Delta_k u^{\sharp}\|_{L^2}\|\Psi(u) - u^{\sharp}\circ\xi\|_{L^2} \lesssim  \|2^{2k}\Delta_k u^{\sharp}\|_{L^2}\|u\|_{H^{\alpha}}.
\end{equation}
By Sobolev embedding, we obtain
\begin{equation}\label{R2455}
    \int_{\mathbb{T}^2} \Gamma(2^{2k}\Delta_k u^{\sharp})a|u|^2 u dx \lesssim \|\Gamma(2^{2k}\Delta_k u^{\sharp})\|_{L^2}\|a|u|^2 u\|_{L^2} \lesssim \|2^{2k}\Delta_k u^{\sharp}\|_{L^2}\|u\|_{H^{\alpha}}.
\end{equation}
We combine above estimates, and sum over $ k$ to get
\begin{align}
    \|u^{\sharp}\|^2_{H^{2-\kappa}} + \|u^{\sharp}\|^2_{H^{1-\kappa}} \lesssim \|u^{\sharp}\|_{H^{2-2\kappa}}(\|u\|_{H^{\alpha}}+\|u^{\sharp}\|_{H^1}).
\end{align}
After using the weighted Young inequality and choosing $\delta$ small enough to absorb $\|u^{\sharp}\|_{H^{2-2\kappa}}$ into the left hand side, we have
\begin{equation}
    \| u^{\sharp} \|^2_{H^{2}} \lesssim \|u\|^2_{H^{\alpha}}+\|u^{\sharp}\|^2_{H^1}.
\end{equation}
The proof is completed.
\end{proof}

\subsection{Schauder estimates}\label{sbSE}

By $L^2$ estimates, the minimizer $u \in \mathscr{D}_{\vartheta}^{\alpha,2}$. Moreover, $u$ can be decomposed into $u = u\prec\vartheta + R(u) + u^{\sharp}$, and these two terms $u\prec \vartheta + R(u)$ and $u^{\sharp}$ satisfy the elliptic system (\ref{decPAM}).
By the elliptic system, we establish the following Schauder estimates for $u$, $R(u)$ ,and $\psi$.

\begin{thm}\label{SEu}
The minimizer $u$ given in Theorem \ref{Ex} satisfies $u \in \mathscr{C}^{\alpha}$, with remainders $R(u)\in \mathscr{C}^{2\alpha}$ and $u^{\sharp}\in \mathscr{C}^{3\alpha}$.
\end{thm}

\begin{proof}
We prove in three steps.\\
{\bf Step 1. Bound for $u$ in $\mathscr{C}^{\alpha}$} \\
First, we estimate $\Phi(u)$ in $ \mathscr{C}^{-2+\kappa}$, and derive a priori for $\phi$ in $ \mathscr{C}^{\kappa} $ by elliptic Schauder estimates.
By Besov embedding, we have $H^{\alpha} = B^{\alpha}_{2,2} \hookrightarrow \mathscr{C}^{\alpha-1}$ and $u \in \mathscr{C}^{\alpha-1}$. Then the paraproduct estimates imply that
\begin{equation}\label{E411}
    \|u \prec \mathscr{U}_{>} \xi\|_{\mathscr{C}^{-2+\kappa}} + \|u \succ \mathscr{U}_{>} \xi\|_{\mathscr{C}^{-2+\kappa}} \lesssim 2^{-(\alpha-2\kappa)} \|\xi\|_{\mathscr{C}^{-1-\kappa}}\|u\|_{\mathscr{C}^{\alpha-1}},
\end{equation}
and
\begin{equation}\label{E412}
    \|u \prec \mathscr{U}_{>} (\vartheta\diamond\xi)\|_{\mathscr{C}^{-2+\kappa}} + \|u \succ \mathscr{U}_{>} (\vartheta\diamond\xi)\|_{\mathscr{C}^{-2+\kappa}}
   \lesssim   2^{-(1-\kappa+\alpha)}\|\vartheta\diamond\xi \|_{\mathscr{C}^{-2\kappa }} \|u\|_{\mathscr{C}^{\alpha-1}}.
\end{equation}
Then by Schauder estimates, we have
\begin{equation}\label{Ephi1E}
    \|u\prec\vartheta + R(u)\|_{\mathscr{C}^{\kappa}}\lesssim \|\Phi(u)\|_{\mathscr{C}^{-2+\kappa}} \lesssim (2^{-(1 -\kappa)L}+2^{-(2-3\kappa)K})\|u\|_{\mathscr{C}^{\alpha-1}} \lesssim \|u\|_{H^{\alpha}}.
\end{equation}
Since $u^{\sharp} \in H^2$, the Sobolev embedding theorem implies that $\psi \in \mathscr{C}^{\alpha}$. Thus $u = u\prec \vartheta + R(u) \in \mathscr{C}^{\kappa}$.
Now we estimate $\Phi(u)$ in $\mathscr{C}^{-2+\alpha}$, and derive a bound for $u\prec \vartheta + R(u)$ in $\mathscr{C}^{\alpha}$ by Schauder estimates.
By Bony's paraproduct estimates, we have
\begin{align}\label{E413}
  \|u \prec \mathscr{U}_{>} \xi + u \succ \mathscr{U}_{>} \xi\|_{\mathscr{C}^{-2+\alpha}} \lesssim   &  \|\mathscr{U}_{>}\xi\|_{\mathscr{C}^{\alpha-2}} \|u\|_{L^{\infty}} \nonumber \\
     \lesssim &  2^{-(1-\kappa-\alpha)L}\|\mathscr{U}_{>}\xi\|_{\mathscr{C}^{-1-\kappa}}\|u\|_{L^{\infty}},
\end{align}
and
\begin{align}\label{E414}
  \|u \succ \mathscr{U}_{>}(\vartheta\circ\xi) + u \prec \mathscr{U}_{>}(\vartheta\circ\xi)\|_{\mathscr{C}^{-2+\alpha}} \lesssim   &  \| \mathscr{U}_{>}(\vartheta\diamond\xi)\|_{\mathscr{C}^{\alpha-2 }} \|u\|_{L^{\infty}} \nonumber \\
     \lesssim &  2^{-(2-2\kappa -\alpha)K}\|\mathscr{U}_{>}(\vartheta\circ\xi)\|_{\mathscr{C}^{-2\kappa}}\|u\|_{L^{\infty}}.
\end{align}
Then by above estimates (\ref{E411})-(\ref{E414}), we have
\begin{equation}\label{phicalphaE}
    \|u\prec\vartheta + R(u)\|_{\mathscr{C}^{\alpha}} = \|\mathscr{L}^{-1}(\Phi(u))\|_{\mathscr{C}^{\alpha}} \lesssim \|\Phi(u)\|_{\mathscr{C}^{-2+\alpha}}  \lesssim  \|u\|_{L^{\infty}} \lesssim \|\phi\|_{\mathscr{C}^{\kappa}}+\|\psi\|_{\mathscr{C}^{\kappa}} \lesssim \|u\|_{\mathscr{D}_{\vartheta}^{\alpha,2}}.
\end{equation}
Thus $u = u\prec\vartheta + R(u) + u^{\sharp} \in \mathscr{C}^{\alpha}$.\\
{\bf Step 2. Bound for $R(u)$ in $\mathscr{C}^{2\alpha}$} \\
Note that $R(u)$ satisfies
\begin{align}\label{usharpgE}
      \mathscr{L}R(u) & =  :- \mathscr{L}(u\prec \vartheta) + \Phi(u)  \nonumber \\
                    & = - \mathscr{L}u\prec\vartheta - \nabla u \prec \nabla\vartheta - u \prec \xi + \Phi(u).
\end{align}
By paraproduct estimates and $ \|\xi\|_{\mathscr{C}^{\alpha-2}} \lesssim 1$, we have
\begin{equation}\label{Es1E}
        \|\mathscr{L}u\prec\vartheta\|_{\mathscr{C}^{2\alpha-2}} \lesssim \|\mathscr{L}u\|_{\mathscr{C}^{\alpha-2}}\|\vartheta\|_{\mathscr{C}^{1+\kappa}} \lesssim \|u\|_{\mathscr{C}^{\alpha} },
\end{equation}
and
\begin{equation}\label{Es2E}
        \|\nabla u \prec \nabla\vartheta\|_{\mathscr{C}^{2\alpha-2}} \lesssim \|\nabla u\|_{\mathscr{C}^{\alpha -1}}\|\nabla\vartheta\|_{\mathscr{C}^{\kappa}} \lesssim \|u\|_{\mathscr{C}^{\alpha} }.
\end{equation}
Recall the choosing of $L,K$, by paraproduct estimates, we get
\begin{align}\label{Es3E}
       & \|-u\prec\xi+\Phi(u)\|_{\mathscr{C}^{2\alpha -2} } \nonumber \\
    \lesssim & \|u \succ \mathscr{U}_{>} \xi + u \succ \mathscr{U}_{>}(\vartheta\circ\xi)\|_{\mathscr{C}^{2\alpha -2} }  +\|u \prec \mathscr{U}_{\leq} \xi  + u \prec \mathscr{U}_{>}(\vartheta\circ\xi)\|_{\mathscr{C}^{2\alpha -2} } \nonumber \\
    \lesssim & (\| \mathscr{U}_{>} \xi\|_{\mathscr{C}^{\alpha-2}} + \| \mathscr{U}_{>} (\vartheta\circ\xi)\|_{\mathscr{C}^{\alpha-2}})\|u\|_{\mathscr{C}^{\alpha} }    +(\| \mathscr{U}_{\leq} \xi\|_{\mathscr{C}^{2\alpha-2}}+ \| \mathscr{U}_{>} (\vartheta\circ\xi)\|_{\mathscr{C}^{2\alpha-2}})\|u\|_{L^{\infty}} \nonumber \\
   \lesssim  & (2^{-(1-\kappa-\alpha)L}+2^{-(2-2\kappa-\alpha)K})\|u\|_{\mathscr{C}^{\alpha} }  + (2^{(2\alpha-1+\kappa )L}+ 2^{-(2-2\kappa-2\alpha)K})\|u\|_{L^{\infty}} \nonumber \\
   \lesssim & \|u\|_{\mathscr{C}^{\alpha} }
\end{align}
Combining with above estimates (\ref{Es1E})-(\ref{Es3E}), and using the Schauder estimates, we have
\begin{align}\label{phisharpE}
    \|R(u)\|_{\mathscr{C}^{2\alpha}} \lesssim  \|- \mathscr{L}u\prec\vartheta - \nabla u \prec \nabla\vartheta - u) \prec \xi + \Phi(u)\|_{\mathscr{C}^{2\alpha -2} }
    \lesssim  \|u\|_{\mathscr{C}^{\alpha} }.
\end{align}
{\bf Step 3. Bound for $u^{\sharp}$ in $\mathscr{C}^{3\alpha}$} \\
We derive a bound for $u^{\sharp}$ in $\mathscr{C}^{3\alpha}$. By paraproduct estimates and a priori estimates (\ref{phicalphaE}), (\ref{phisharpE}), we have
\begin{equation}\label{E41E}
    \|R(u) \circ \xi\|_{\mathscr{C}^{3\alpha -2}} \lesssim  \|\xi\|_{\mathscr{C}^{-1-\kappa}}\|R(u)\|_{\mathscr{C}^{2\alpha}}
    \lesssim \|u\|_{\mathscr{C}^{\alpha} }
\end{equation}
\begin{equation}\label{E42E}
    \|u^{\sharp}\circ\xi\|_{\mathscr{C}^{3\alpha -2}}  \lesssim \|\psi\|_{\mathscr{C}^{2\alpha }},
\end{equation}
\begin{equation}\label{E43E}
    \|\mathscr{U}_{\leq}(\vartheta\diamond\xi) \prec u\|_{\mathscr{C}^{3\alpha -2}}
    \lesssim  \|\vartheta\diamond\xi\|_{\mathscr{C}^{2\alpha -2}}\|u\|_{\mathscr{C}^{\alpha}} \lesssim \|u\|_{\mathscr{C}^{\alpha} }.
\end{equation}
\begin{equation}\label{E44E}
    \|(\vartheta\diamond\xi) \circ u\|_{{\mathscr{C}^{3\alpha -2}}}
    \lesssim  \|u\|_{\mathscr{C}^{\alpha}}.
\end{equation}
The commutator estimate Lemma \ref{commutatorE} implies that
\begin{equation}
    \|C(u,\vartheta,\xi)\|_{{\mathscr{C}^{3\alpha-2}}} \lesssim \|u\|_{{ \mathscr{C}^{\alpha}}} \|\xi\|_{\alpha -2}\|\vartheta\|_{\alpha} \lesssim \|u\|_{\mathscr{C}^{\alpha}}.
\end{equation}
According to Lemma \ref{Localization}, we have
\begin{align}
     & \|u\prec \mathscr{U}_{\leq}(\vartheta\circ\xi) \|_{\mathscr{C}^{3\alpha-2} }+ \|u\succ \mathscr{U}_{\leq} \xi \|_{\mathscr{C}^{3\alpha-2} } \nonumber \\
     \lesssim & \|u\|_{L^{\infty}}(\|\mathscr{U}_{\leq} (\vartheta\circ\xi)\|_{\mathscr{C}^{3\alpha-2}}+ \|\mathscr{U}_{\leq} \xi\|_{\mathscr{C}^{3\alpha-2}}) \nonumber \\
     \lesssim & 2^{(3\alpha-1+\kappa)L}\|\xi\|_{\mathscr{C}^{-1-\kappa}}\|u\|_{L^{\infty}} + 2^{(3\alpha-2+2\kappa)K}\|\vartheta\circ\xi\|_{\mathscr{C}^{-2\kappa}}\|u\|_{L^{\infty}},
\end{align}
and
\begin{align}
       & \|u \succ \mathscr{U}_{\leq}(\vartheta\circ\xi)\|_{\mathscr{C}^{3\alpha-2 }}+ \|u \succ \mathscr{U}_{\leq}\xi\|_{\mathscr{C}^{3\alpha-2 }} \nonumber \\
        \lesssim & (\|\mathscr{U}_{\leq}(\vartheta\circ\xi)\|_{\mathscr{C}^{2\alpha-2}} + \|\mathscr{U}_{\leq}\xi\|_{\mathscr{C}^{2\alpha-2}}) \|u\|_{{ \mathscr{C}^{\alpha}}} \nonumber\\
        \lesssim & (2^{(2\alpha-1+\kappa)L}+1)\|u\|_{{ \mathscr{C}^{\alpha}}}.
\end{align}
By (\ref{Ephi1E}), we have
\begin{align}\label{E48E}
    a\||u|^2 u\|_{\mathscr{C}^{3\alpha-2 }} \lesssim \|u\|^2_{L^{\infty}}\|u\|_{\mathscr{C}^{\alpha}}.
\end{align}
Combining with above estimates (\ref{E41E})-(\ref{E48E}), and using the interpolation inequality in Lemma \ref{interpolation} and weighted Young inequality, for every $\delta>0$ we have
\begin{equation*}
    \|\Psi(u)\|_{\mathscr{C}^{3\alpha-2 }} \lesssim \|u\|_{\mathscr{C}^{\alpha}} + \|u^{\sharp}\|^{2/3}_{\mathscr{C}^{3\alpha}}\|u^{\sharp}\|^{1/3}_{L^{\infty}} \lesssim \|u\|_{\mathscr{C}^{\alpha}} + \delta\|u^{\sharp}\|_{\mathscr{C}^{3\alpha}} + C_{\delta}\|u^{\sharp}\|_{\mathscr{C}^{\alpha}}.
\end{equation*}
Then by Schauder estimates and choosing $\delta$ small enough, we obtain
\begin{align}\label{psiE}
\|u^{\sharp}\|_{\mathscr{C}^{3\alpha}} = \|\mathscr{L}^{-1}(a|u|^2 u + \Psi(u))\|_{\mathscr{C}^{3\alpha}}\lesssim  1+ \|u^{\sharp}\|_{\mathscr{C}^{\alpha}} + \|u \|_{\mathscr{C}^{\alpha}} +\|u^{\sharp}\|^{2}_{L^{\infty}}\|u\|_{\mathscr{C}^{\alpha}}.
\end{align}
The proof is completed.
\end{proof}

\section{The tail estimate of the principle eigenvalue}

In this section, we turn to prove the tail estimate for the principle eigenvalue as Theorem \ref{TEGS}.
Recall that principle eigenvalue is given by following $L^2$ constrained minimization problem:
\begin{equation}\label{L2ME}
    \lambda = \inf_{w \in \mathscr{D}^{\alpha,1}_{\vartheta}, \|w\|_{L^2}=1} E(w) = \frac{1}{2} B_{\mathscr{H}}(u,u) + \frac{a}{4}\int_{\mathbb{T}^2} |u|^4 dx,
\end{equation}
where $u$ is a minimizer of the above $L^2$ constrained minimization problem which satisfies the stationary nonlinear Schr\"{o}dinger equation (\ref{NAME}). By energy estimates, we obtain the upper bound and lower bound as follows.

\begin{proof} \textbf{Upper bound:}\\
We choose the first eigenfunction $e_{1}$ of the Anderson hamiltonian $\mathscr{H}$ as a test function in the energy functional $E(u)$, and use Theorem \ref{ThmBh} to estimate
\begin{align}
    \lambda = & \frac{1}{2} B_{\mathscr{H}}(u,u) + \frac{a}{4}\int_{\mathbb{T}^2} |u|^4 dx \nonumber \\
            \leq & \frac{1}{2} B_{\mathscr{H}}(e_{1},e_{1}) + \frac{a}{4}\int_{\mathbb{T}^2} |e_{1}|^4 dx \nonumber \\
            \leq & \Lambda_{1} + \|e_{1}\|^2_{\mathscr{D}_{\vartheta}^{\alpha,1}} \nonumber \\
            \leq & C'\Lambda_{1} + C_{\xi},
\end{align}
where the deterministic constant $C'$ is independent with $\|\xi\|_{\mathscr{C}^{-1-\kappa}}$ and $\|\vartheta\diamond\xi\|_{\mathscr{C}^{-2\kappa}}$.
It follows that 
\begin{equation}
    \mathbb{P}( \lambda \leq -x )  \leq  \mathbb{P}(\Lambda_{1} \leq -\frac{x}{C'} ).
\end{equation}
Then by the tail estimate (\ref{TElam}) for $\Lambda_{1}$, we obtain the upper bound.\\
\textbf{Lower bound:}\\
In order to get the lower bound, we choose $e_{1}$ as a test function again, and use the identity (\ref{L2ME}) to estimate
\begin{align}
    \lambda = & \frac{1}{2} B_{\mathscr{H}}(u,u) + \frac{a}{4} \int_{\mathbb{T}^2} |u|^4 dx \nonumber \\
            \geq & \frac{1}{2} B_{\mathscr{H}}(u,u) \nonumber \\
            \geq & \frac{1}{2} B_{\mathscr{H}}(e_{1},e_{1}) \nonumber \\
            = & \Lambda_{1}
\end{align}
Then combining with the tail estimate for $\Lambda_{1}$, we obtain the lower bound of $\lambda$. The proof is completed.
\end{proof}

 \section{Conclusion}\label{DC}

This paper is an attempt to build a bridge between the variation problem and the singular stochastic partial differential equation in the paracontrolled distribution framework. We define the energy functional $E(u)$ associated with the Anderson hamiltonian on the suitable energy space $\mathscr{D}_{\vartheta}^{\alpha,1}$ with paracontrolled distributions structure . Then we show that the energy functional $E(u)$ is a $C^1$ map from $\mathscr{D}_{\vartheta}^{\alpha,1}$ to $\mathbb{R}$, and the Euler-Lagrange equation of the energy functional $E(u)$ is the elliptic singular stochastic partial differential equation (\ref{NAME}). By the direct method of calculus of  variation, we proved the existence of minimizers. Since the minimizer $u$ is a weak solution of the elliptic singular stochastic partial differential equation (\ref{NAME}), we use the structure of the singular stochastic partial differential equation (\ref{NAME}), and establish the $L^2$ estimates and Schauder estimates for the minimizer $u$.

We restrict our study to the Anderson hamiltonian in $2$-dimensional torus $\mathbb{T}^2$ with the periodic boundary condition. By Dirichlet and Neumann Besov spaces, we can replace the periodic boundary condition by Dirichlet or Neumann boundary condition. In the $3$-dimensional case, the regularity of spatial white noise is much more singular, which makes the definition of the Anderson hamiltonian and its energy space more complex. This will be the subject of future work.

\section*{Appendix}

\appendix

\section{Besov space and Bony's paraproduct}

In this Appendix, we recall some basic notations and useful estimates about Littlewood-Paley decomposition, Besov space and Bony's paraproduct. For more details, we refer to \cite{BCD2011,GIP2015, GH2019}.

Littlewood-Paley decomposition can describe the regularity of (general) functions via the decomposition of a (general) function into a series of smooth functions with different frequencies. In order to do this, we introduce the following dyadic partition. 

There exist two smooth radial functions $\chi$ and $\varrho$, valued in the interval $[0,1]$, and so that
\begin{enumerate}
\item $supp(\chi) \subset B_{4/3}(0)$ and $supp(\varrho)\subset \{x \in \mathbb{R}^d: \frac{3}{4}\leq |x| \leq \frac{8}{3} \}$;
\item $\chi(x)+\sum_{j\geq 0}\varrho(2^{-j}x)=1, \quad x\in \mathbb{R}^n$;
\item $supp(\chi)\cap supp(\varrho(2^{-j}x)) = \emptyset $ for $j\geq 1$ and $supp(\varrho(2^{-i}x)) \cap supp(\varrho(2^{-j}x)) = \emptyset $ for $|i-j|\geq 2$.
\end{enumerate}

\begin{defn}
For $u \in \mathcal{S}'(\mathbb{T}^d)$ and $j\geq -1$, the Littlewood-Paley blocks of $u$ are defined as
\begin{equation*}
    \Delta_j u = \mathscr{F}^{-1}_{\mathbb{T}^d}(\varrho_j \mathscr{F}_{\mathbb{T}^d}u),
\end{equation*}
where $\varrho_{-1}=\chi$ and $\varrho_j=\varrho(2^{-j}\cdot)$ for $j\geq 0$.
\end{defn}
Now we define the Besov space $B^{\alpha}_{p,q}$ via the Littlewood-Paley blocks as follows.
\begin{defn}
For $\alpha\in\mathbb{R}$, $p,q \in [1,\infty]$, we define Besov space
\begin{equation*}
    B^{\alpha}_{p,q}(\mathbb{T}^d) = \left\{ u\in \mathcal{S}'(\mathbb{T}^d): \|u\|_{B^{\alpha}_{p,q}(\mathbb{T}^d)}=  	\left( \sum_{j \geq -1}(2^{j\alpha}\|\Delta_j u\|_{L^p(\mathbb{T}^d)})^q \right)^{1/q} < \infty \right\}.
\end{equation*}
\end{defn}

For $\alpha\in\mathbb{R}$, the H\"{o}lder-Besov space on $\mathbb{T}^d$ is denoted by $\mathscr{C}^{\alpha}=B^{\alpha}_{\infty,\infty}(\mathbb{T}^d)$. We remark that if $\alpha \in (0,\infty)\backslash\mathbb{N}$, then the H\"{o}lder-Besov space $\mathscr{C}^{\alpha}$ is equal to the H\"{o}lder space $C^{\alpha}(\mathbb{T}^d)$. The Sobolev space $H^{\alpha}$ is the same as the Besov space $B^{\alpha}_{2,2}(\mathbb{T}^d)$.

We need the following Bernstein inequality in $L^2$ estimates.
\begin{lem}\label{BerI}
Let $\mathscr{B}$ ba a unit ball, $n \in \mathbb{N}_0$, and $1\leq p \leq q \leq \infty$. Then for every $\lambda >0$ and $u \in L^p$ with $supp(\mathscr{F}u) \subset \lambda \mathscr{B}$, we have
\begin{equation*}
    \max_{1 \in \mathbb{N}^d: |1|=n}\|\partial_{1} u\|_{L^q} \lesssim C_{n,p,q,\mathscr{B}}\lambda^{n+d(\frac{1}{p}-\frac{1}{q})}\|u\|_{L^p}.
\end{equation*}
\end{lem}

The Besov embedding theorem is useful in regularity estimates.
\begin{lem}\label{Besovem}
Let $1 \leq p_1 \leq p_2 \leq \infty$, $1 \leq q_1 \leq q_2 \leq \infty $, and $\alpha \in\mathbb{R}$. Then we have
\begin{equation*}
    B^{\alpha}_{p_1, q_1}(\mathbb{T}^d) \hookrightarrow B^{\alpha -d(1/p_1 -1/p_2)}_{p_2, q_2}(\mathbb{T}^d).
\end{equation*}
\end{lem}

Now we introduce the Bony's paraproduct. Let $f$ and $g$ be tempered distributions in $\mathcal{S}^{\prime}(\mathbb{T}^d)$. By Littlewood-Paley blocks, the product $fg$ can be (formally) decomposed as 
\begin{equation*}
    fg = \sum_{j\geq -1}\sum_{i\geq -1} \Delta_{i} f \Delta_{j} g = f\prec g +f\circ g+ f\succ g,
\end{equation*}
where
\begin{equation*}
    f\prec g = g\succ f= \sum_{j\geq -1}\sum_{i =-1}^{j-2} \Delta_{i} f \Delta_{j} g \quad \text{and} \quad f\circ g = \sum_{|i-j|\leq 1} \Delta_{i} f \Delta_{j} g.
\end{equation*}
We have following paraproduct estimates in the Bony's paraproduct (See \cite[Lemma 2.1]{GIP2015} and \cite[Proposition A.1]{GUZ2020}).
\begin{lem}\label{Bparaproduct}
For every $\beta \in \mathbb{R}$, we have
\begin{equation*}
    \|f\prec g\|_{\mathscr{C}^{\beta}} \lesssim \|f\|_{L^{\infty}}\|g\|_{\mathscr{C}^{\beta}},
\end{equation*}
\begin{equation*}
    \|f \prec g\|_{H^{\beta}} \lesssim \|f\|_{L^2} \|g\|_{\mathscr{C}^{\beta + \kappa }} \wedge \|f\|_{L^{\infty}}\|g\|_{H^{\beta}} \quad \text{for all }\kappa >0.
\end{equation*}
If $\beta \in \mathbb{R}$, $\alpha <0$, we have
\begin{equation*}
    \|f\prec g\|_{\mathscr{C}^{\alpha+\beta}} \lesssim \|f\|_{\mathscr{C}^{\alpha}}\|g\|_{\mathscr{C}^{\beta}},
\end{equation*}
\begin{equation*}
    \|f \prec g\|_{H^{\alpha + \beta }} \lesssim \|f\|_{H^{\alpha}}\|g\|_{\mathscr{C}^{\beta + \kappa }} \wedge \|f\|_{\mathscr{C}^{\alpha}}\|g\|_{H^{\beta}} \quad \text{for all }\kappa >0.
\end{equation*}
Moreover, if $\alpha+\beta>0$, then
\begin{equation*}
\|f\circ g\|_{\mathscr{C}^{\alpha +\beta}} \lesssim \|f\|_{\mathscr{C}^{\alpha}}\|g\|_{\mathscr{C}^{\beta}},
\end{equation*}
\begin{equation*}
    \|f\circ g\|_{H^{\alpha + \beta}} \lesssim \|f\|_{\mathscr{C}^{\alpha}}\|g\|_{H^{\beta}}.
\end{equation*}
\end{lem}

The following commutator estimate is also crucial in paracontrolled distribution (See \cite[Lemma 2.4]{GIP2015} and \cite[Proposition A.2]{GUZ2020})

\begin{lem}\label{commutatorE}
. Assume that $\alpha\in (0,1)$ and $\beta, \gamma \in \mathbb{R}$ are such that $\alpha+\beta +\gamma>0$ and $\beta +\gamma <0$. Then for $u,v,h \in C^{\infty}(\mathbb{T}^d)$, the trilinear operator
\begin{equation*}
    C(u,v,h)=(u\prec v)\circ h-u(v\circ h)
\end{equation*}
has the following estimate
\begin{equation*}
    \|C(u,v,h)\|_{\mathscr{C}^{\alpha+\beta +\gamma}} \lesssim \|u\|_{\mathscr{C}^{\alpha}}\|v\|_{\mathscr{C}^{\beta}}\|h\|_{\mathscr{C}^{\gamma}}.
\end{equation*}
Thus $C$ can be uniquely extended to a bounded trilinear operator from $\mathscr{C}^{\alpha}\times \mathscr{C}^{\beta} \times \mathscr{C}^{\gamma}$ to $\mathscr{C}^{\alpha+\beta +\gamma}$. For $H^{\alpha}$ space, we also have
\begin{equation*}
    \|C(u,v,h)\|_{H^{\alpha + \beta +\gamma}} \lesssim \|u\|_{ H^{\alpha}}\|v\|_{H^{\beta}}\|h\|_{\mathscr{C}^{\gamma }}.
\end{equation*}
It implies that $C$ can be uniquely extended to a bounded trilinear operator from $H^{\alpha}\times H^{\beta}\times \mathscr{C}^{\gamma}$ to $H^{\alpha+\beta +\gamma}$.
\end{lem}

For every $u,v,h \in C^{\infty}(\mathbb{T}^d)$, we define the trilinear operator
\begin{equation}
    D(u,v,h) = \langle u, h\circ v \rangle - \langle u\prec v, h\rangle.
\end{equation}
We have the following estimate from \cite[Lemma A.6]{GUZ2020}.
\begin{lem}\label{Dm}
Let $\alpha\in (0,1)$, $\beta, \gamma \in \mathbb{R}$ such that $\alpha + \beta + \gamma > 0$ and $\beta + \gamma < 0$. Then we have
\begin{equation*}
    |D(u,v,h)| \lesssim \|u\|_{ H^{\alpha}}\|v\|_{H^{\beta}}\|h\|_{\mathscr{C}^{\gamma}}.
\end{equation*}
Thus $D$ can be uniquely extended to a bounded trilinear operator from $H^{\alpha}\times H^{\beta}\times \mathscr{C}^{\gamma}$ to $\mathbb{R}$.
\end{lem}

The following estimate from \cite[Proposition A.2]{AC2015} is useful in this paper.
\begin{lem}\label{CE}
Let $f\in H^{\alpha}$, $g \in \mathscr{C}^{\beta}$ with $\alpha \in (0,1)$, $\beta \in \mathbb{R}$. Then
\begin{equation*}
    \|\mathscr{L}(f\prec g)- f\prec (\mathscr{L}g)\|_{H^{\alpha+\beta+2}} \lesssim \|f\|_{H^{\alpha}}\|g\|_{\mathscr{C}^{\beta}}.
\end{equation*}
\end{lem}

We also need the following interpolations result for Besov space.
\begin{lem}\label{interpolation}
Let $\theta \geq 0$, and $u^{\sharp} \in \mathscr{C}^{\gamma}$. Then for any $\alpha \in [0, \gamma]$, we have
\begin{equation}
    \|u^{\sharp}\|_{\mathscr{C}^{\alpha}} \lesssim \| u^{\sharp}\|_{L^{\infty}}^{1-\alpha/\gamma}  \| u^{\sharp}\|_{\mathscr{C}^{\gamma}}^{\alpha/\gamma}.
\end{equation}
\end{lem}
\begin{proof}
It holds 
\begin{align*}
    \|\Delta_k u^{\sharp}\|_{L^{\infty}} \lesssim & \|\Delta_k u^{\sharp}\|_{L^{\infty}}^{1-\alpha/\gamma}\|\Delta_k u^{\sharp}\|_{L^{\infty}}^{\alpha/\gamma} \\
     \lesssim & 2^{-\alpha k} \|\Delta_k u^{\sharp}\|^{1-\alpha/\gamma}_{L^{\infty}}  \| u^{\sharp}\|_{\mathscr{C}^{\gamma}}^{\alpha/\gamma}.
\end{align*}     
Thus we obtain
\begin{equation*}
    \|u^{\sharp}\|_{\mathscr{C}^{\alpha}} \lesssim \| u^{\sharp}\|_{L^{\infty}}^{1-\alpha/\gamma}  \| u^{\sharp}\|_{\mathscr{C}^{\gamma}}^{\alpha/\gamma}.
\end{equation*}
This completes the proof.
\end{proof}

\begin{lem}\label{interpolationH}
Let $\beta \in (0,1)$ and $u^{\sharp} \in H^{\beta}$. Then  for arbitrary $\delta>0$, we have
\begin{equation}
    \|u^{\sharp}\|^2_{H^{\beta}} \lesssim \delta \| \nabla u^{\sharp} \|_{L^2}^2 + C_{\delta} \|u^{\sharp} \|_{L^2}^2.
\end{equation}
\end{lem}
\begin{proof}
Since $\|u^{\sharp}\|_{H^{\beta}} \simeq \|u^{\sharp}\|_{B_{2,2}^{\beta}}$, by Bernstein inequality (Lemma \ref{BerI}), H\"{o}lder inequality and weighted Young inequality, we have
\begin{align}
    \|u^{\sharp}\|_{H^{\beta}} & =\sum_{i\geq -1} 2^{2 \beta k } \|\Delta_{i} u^{\sharp}\|^2_{L^2} \nonumber \\
    & = \sum_{i\geq -1} 2^{2 \beta k } \|\Delta_{i} u^{\sharp}\|^{2\beta}_{L^2} \|\Delta_{i} u^{\sharp}\|^{2(1-\beta)}_{L^2} \nonumber \\
    & \leq \left( \sum_{i\geq -1}  2^{2  k }\|\Delta_{i}u^{\sharp}\|^{2}_{L^2} \right)^{\beta} \left( \sum_{i\geq -1}  \|\Delta_{i}u^{\sharp}\|^{2}_{L^2} \right)^{1-\beta} \nonumber \\
    & \lesssim\| \nabla u^{\sharp} \|_{L^2}^{2\beta} \|u^{\sharp}\|^{2(1-\beta)}_{L^2}  \nonumber \\
    & \lesssim \delta \| \nabla u^{\sharp} \|_{L^2}^2 + C_{\delta} \|u^{\sharp} \|_{L^2}^2.
\end{align}
This completes the proof.
\end{proof}

\section{A proof of the self-adjointness of Anderson hamiltonian}

In this Appendix, we provide a self-contained proof of the self-adjointness of Anderson hamiltonian as Theorem \ref{sfH} for completeness. Before prove the main result, we first show that $\mathscr{H}$ is a linear bounded operator from $\mathscr{D}_{\vartheta}^{\alpha,2}$ to $L^2$ in the following Lemma.

\begin{lem}\label{Hcov}
The Anderson hamiltonian $\mathscr{H}$ is a linear bounded operator from its domain $\mathscr{D}_{\vartheta}^{\alpha,2}$ to $L^2$ with estimates
\begin{equation}
    \|\mathscr{H}u\|_{L^2} \lesssim  \|u^{\sharp}\|_{H^2} +  \|u\|_{H^{\alpha}} \lesssim \|u\|_{\mathscr{D}_{\vartheta}^{\alpha,2}},
\end{equation}
and
\begin{equation}
    \|u^{\sharp}\|_{H^2} \lesssim \|\mathscr{H}u\|_{L^2} + (1 + \|\xi\|^2_{\mathscr{C}^{-1-\kappa}} + \|\xi\|_{\mathscr{C}^{-1-\kappa}}\|\vartheta\diamond\xi\|_{\mathscr{C}^{-2\kappa}}) \|u\|_{L^2}.
\end{equation}
\end{lem}

\begin{proof}
For every $u\in \mathscr{D}_{\vartheta}^{\alpha,2}$, by estimate (\ref{Epsi}) and paraproduct estimates, we have
\begin{equation}\label{Epsi2}
    \|\Psi(u)\|_{L^2} \leq \|\Psi(u) - u^{\sharp}\circ\xi\|_{L^2} +\|u^{\sharp} \circ\xi\|_{L^2} \lesssim \|u\|_{H^{\alpha}} + \|u^{\sharp}\|_{H^{1+\kappa}}.
\end{equation}
It follows that
\begin{equation}
    \|\mathscr{H}u\|_{L^2} \leq \|\mathscr{L}u^{\sharp}\|_{L^2} + \|\Psi(u)\|_{L^2} \lesssim  \|u^{\sharp}\|_{H^2} +  \|u\|_{H^{\alpha}} \lesssim \|u\|_{\mathscr{D}_{\vartheta}^{\alpha,2}},
\end{equation}
Thus the Anderson hamiltonian $\mathscr{H}$ is a linear bounded operator from its domain $\mathscr{D}_{\vartheta}^{\alpha,2}$ to $L^2$. By interpolation inequality, weighted Young inequality, and the operator $\Gamma$, we obtain
\begin{align*}
    \|\mathscr{H}u - \mathscr{L}u^{\sharp}\|_{L^2} = & \|\Psi(u)\|_{L^2} \\
    \lesssim & (1+\|\xi\|^2_{\mathscr{C}^{-1-\kappa}} + \|\vartheta\diamond\xi\|^2_{\mathscr{C}^{-2\kappa}}) \| \Gamma u^{\sharp}\|_{H^{\alpha}} + \|\xi\|_{\mathscr{C}^{-1-\kappa}}\|u^{\sharp}\|_{H^{1+\kappa}}\\
    \lesssim & (1+\|\xi\|^2_{\mathscr{C}^{-1-\kappa}} + \|\vartheta\diamond\xi\|^2_{\mathscr{C}^{-2\kappa}}) \| u^{\sharp}\|_{H^{\alpha}} + \|\xi\|_{\mathscr{C}^{-1-\kappa}}\|u^{\sharp}\|_{H^{1+\kappa}}\\
    \leq & C_{\delta}(1 + \|\xi\|^2_{\mathscr{C}^{-1-\kappa}} + \|\vartheta\diamond\xi\| ^2_{\mathscr{C}^{-2\kappa}})\|u^{\sharp}\|_{L^2}+\delta\|u^{\sharp}\|_{H^2}.
\end{align*}
where $C_{\delta}, \delta$ are positive constant which are independent with $\|\xi\|_{\mathscr{C}^{-1-\kappa}}$ and $\|\vartheta\diamond\xi\|_{\mathscr{C}^{-2\kappa}}$. 
Note that $\|u^{\sharp}\|_{L^2} \lesssim \|u\|_{L^2}$.
We choose $\delta$ small enough to absorb $\delta \|u^{\sharp}\|_{H^2}$, and obtain
\begin{align}
    \|u^{\sharp}\|_{H^2} \lesssim & \|\mathscr{L}u^{\sharp}\|_{L^2} \nonumber \\
    \lesssim & \|\mathscr{H} u\|_{L^2}  +  \|\mathscr{H}u-\mathscr{L}u^{\sharp}\|_{L^2} \nonumber \\
    \lesssim & \|\mathscr{H}u\|_{L^2} + C_{\delta}(1 + \|\xi\|^2_{\mathscr{C}^{-1-\kappa}} + \|\xi\|_{\mathscr{C}^{-1-\kappa}}\|\vartheta\diamond\xi\|_{\mathscr{C}^{-2\kappa}}) \|u\|_{L^2}.
\end{align}
The proof is completed.
\end{proof}

We turn to prove that the Anderson hamiltonian $\mathscr{H}$ is a closed and symmetric operator.

\begin{lem}
The Anderson hamiltonian $\mathscr{H}:\mathscr{D}_{\vartheta}^{\alpha,2}\rightarrow L^2$ is a closed and symmetric operator on $L^2$. Moreover, for every $u \in \mathscr{D}_{\vartheta}^{\alpha,2}$, $\mathscr{H}u$ can be approximated as
\begin{equation}
    \lim_{\epsilon \rightarrow 0} \|\mathscr{H}u- \mathscr{H}_{\epsilon}u_{\epsilon}\|_{L^2} =0,
\end{equation}
where $\mathscr{H}_{\epsilon}:=\mathscr{L} - \xi_{\epsilon} - C_{\epsilon}$ is a self-adjoint operator from $H^2$ to $L^2$, and $u_{\epsilon}= \Gamma_{\epsilon}(u^{\sharp}) \in H^2$.
\end{lem}
\begin{proof}
At first, we verify that the Anderson hamiltonian $\mathscr{H}$ is closed  on its dense domain $ \mathscr{D}_{\vartheta}^{\alpha,2} \subset L^2$. Suppose that $(u_n)_{n\geq 1} \subset \mathscr{D}_{\vartheta}^{\alpha,2}$ with $ u^{\sharp}_n:= u_n- u_n\prec\vartheta - R(u_n)$, such that
\begin{equation*}
    \lim_{n \rightarrow \infty}\|u_n -u\|_{L^2} = 0, \quad \lim_{n \rightarrow \infty}\|\mathscr{H}u_n- f \|_{L^2}=0
\end{equation*}
for some $u, f \in L^2$ . Then by Lemma \ref{Hcov}, $(u^{\sharp}_n)_{n\geq 1} $ is a Cauchy sequence in $H^2$, and $\lim_{n\rightarrow \infty}\|u^{\sharp}_n - u^{\sharp}\|_{H^2}=0$ for some $u^{\sharp}\in H^2$ such that $\Gamma(u^{\sharp}) = u  \in \mathscr{D}_{\vartheta}^{\alpha,2}$. By, we obtain
\begin{align}
    \|\mathscr{H}u_n - f\|_{L^2} \leq & \lim_{n\rightarrow \infty}\|\mathscr{H}(u - u_n)\|_{L^2}+\lim_{n\rightarrow \infty}\|\mathscr{H}u_{n}-f\|_{L^2} \nonumber \\
    \lesssim & \lim_{n\rightarrow \infty}\|u^{\sharp} - u^{\sharp}_n\|_{H^2} + \lim_{n\rightarrow \infty}\|u - u_n\|_{L^2}+\lim_{n\rightarrow \infty}\|\mathscr{H}u_{n}-f\|_{L^2} \nonumber \\
    = & 0.
\end{align}
Thus the Anderson hamiltonian $\mathscr{H}$ is closed.

Then we approximate $\mathscr{H}$ by the self-adjoint operator $\mathscr{H}^{\epsilon}=\mathscr{L} - \xi_{\epsilon} - C_{\epsilon}$. For every $u \in \mathscr{D}_{\vartheta}^{\gamma,2}$, we set $u_{\epsilon}=\Gamma_{\epsilon}(u^{\sharp})$. Then $u_{\epsilon}= u_{\epsilon}\prec\vartheta_{\epsilon} - R_{\epsilon}(u_{\epsilon}) + u^{\sharp} \in H^2$. Further, Lemma \ref{Gam} implies that $u$ can be approximated as
\begin{equation}\label{uApp}
    \lim_{\epsilon\rightarrow 0}\|u - u_{\epsilon}\|_{H^{\gamma}} \lesssim \lim_{\epsilon\rightarrow 0}\|\Gamma u^{\sharp} -\Gamma_{\epsilon} u^{\sharp}\|_{H^{\alpha}} = 0.
\end{equation}
Moreover, $\mathscr{H}_{\epsilon}u_{\epsilon} $ can be written as
\begin{align}
    \mathscr{H}_{\epsilon}u_{\epsilon} = & \mathscr{L} u_{\epsilon} - \xi_{\epsilon}u_{\epsilon} - C_{\epsilon}u_{\epsilon} \nonumber \\
                 = & \mathscr{L} u_{\epsilon}  - u_{\epsilon}\diamond\xi_{\epsilon} \nonumber \\
                 = & \mathscr{L}u^{\sharp} - u^{\sharp}\circ\xi_{\epsilon} - R_{\epsilon}(u_{\epsilon})\circ\xi_{\epsilon}  - u_{\epsilon} \prec \mathscr{U}_{\leq} \xi_{\epsilon} - u_{\epsilon}\succ \mathscr{U}_{\leq} \xi_{\epsilon} - u_{\epsilon}\succ \mathscr{U}_{\leq}(\vartheta_{\epsilon}\circ\xi_{\epsilon})  \nonumber\\
                 & - u_{\epsilon} \prec \mathscr{U}_{\leq}(\vartheta_{\epsilon}\diamond\xi_{\epsilon})- C(u_{\epsilon},\vartheta_{\epsilon},\xi_{\epsilon}) -u_{\epsilon}\circ(\vartheta_{\epsilon} \diamond \xi_{\epsilon}) \nonumber \\
                 := & \mathscr{L} u^{\sharp} - \Psi_{\epsilon}(u_{\epsilon}).
\end{align}
Note that $\xi_{\epsilon} \rightarrow \xi$ in $\mathscr{C}^{-1-\kappa}$, $\vartheta_{\epsilon}\rightarrow \vartheta$ in $\mathscr{C}^{1-\kappa}$, and $\vartheta_{\epsilon}\diamond\xi_{\epsilon} \rightarrow \vartheta\diamond\xi$ in $\mathscr{C}^{-2\kappa}$ as $\epsilon \rightarrow 0$. Similar estimates for $R(u)$ and $R_{\epsilon}(u)$, we have approximation $\lim_{\epsilon \to 0}\|\Psi - \Psi_{\epsilon}\|_{L(H^{\alpha}, L^2)}=0$.
By estimate (\ref{Epsi}) and (\ref{uApp}), it follows that
\begin{align}\label{G0E}
    & \lim_{\epsilon \to 0}\|\mathscr{H}u- \mathscr{H}_{\epsilon}u_{\epsilon}\|_{L^2} \nonumber \\
    = & \lim_{\epsilon \to 0}\|\Psi(u) - \Psi_{\epsilon}(u_{\epsilon})\|_{L^2} \nonumber \\
    \leq & \lim_{\epsilon \to 0}\|\Psi(u - u_{\epsilon})\|_{L^2} + \lim_{\epsilon \to 0}\|\Psi(u_{\epsilon}) - \Psi_{\epsilon}(u_{\epsilon})\|_{L^2} \nonumber \\
    \lesssim &  (1+\|\xi\|^2_{\mathscr{C}^{-1-\kappa}} + \|\vartheta\diamond\xi\|^2_{\mathscr{C}^{-2\kappa}}) \lim_{\epsilon \to 0}\| u - u_{\epsilon}\|_{H^{\alpha}} + \lim_{\epsilon \to 0}\|\Psi - \Psi_{\epsilon}\|_{L(H^{\alpha}, L^2)} \|u_{\epsilon}\|_{H^{\alpha}} \nonumber\\
    = & 0.
\end{align}
By approximations (\ref{uApp}) and (\ref{G0E}), for every $u, v \in \mathscr{D}_{\vartheta}^{\gamma,2}$ we have
\begin{equation}
    \langle v, \mathscr{H}u\rangle = \lim_{\epsilon \rightarrow 0}\langle v_{\epsilon}, \mathscr{H_{\epsilon}}u_{\epsilon} \rangle = \lim_{\epsilon \rightarrow 0}\langle u_{\epsilon}, \mathscr{H_{\epsilon}}v_{\epsilon} \rangle = \langle u, \mathscr{H}v\rangle.
\end{equation}
Thus the Anderson hamiltonian $\mathscr{H}$ is symmetric on $L^2$.
\end{proof}

Now we prove the self-adjointness of the operator $\mathscr{H} + C_{\xi} I$.

\begin{proof}
(Proof of Theorem \ref{sfH}) By Theorem \ref{ThmBh}, we have $c\|u\|^2_{\mathscr{D}_{\vartheta}^{\alpha,1}} \leq \langle u, (\mathscr{H}+ C_{\xi} I)u \rangle \leq C \|u\|^2_{\mathscr{D}_{\vartheta}^{\alpha,1}}$ for every $ u \in \mathscr{D}_{\vartheta}^{\alpha,2}$.
Then by Friedrichs extension theorem (see e.g. \cite[Theorem XI. 7.2]{Y1995}), the operator $\mathscr{H}+\lambda I$ is a positive self-adjoint operator from $\mathscr{D}_{\vartheta}^{\alpha,1}$ to $L^2$.
\end{proof}

\end{document}